\documentclass[leqno,12pt]{amsart}

\usepackage{amssymb}

\usepackage[utf8]{inputenc}

\usepackage{amsthm}
\usepackage{amsmath}
\usepackage{color,comment}

\usepackage{tikz}									
\usetikzlibrary{matrix}
\usetikzlibrary{patterns}
\usetikzlibrary{matrix}
\usetikzlibrary{positioning}
\usetikzlibrary{decorations.pathmorphing}
\usetikzlibrary{cd}

\usepackage[all]{xy}
\usepackage[margin=1in]{geometry}
\usepackage{verbatim}
\usepackage{mathrsfs}
\usepackage{enumitem}
\usepackage[colorlinks=true,allcolors=blue]{hyperref}
\usepackage{amssymb}
\usepackage{mathtools}
\usepackage{comment}

\newtheorem{theorem}{Theorem}[section]
\newtheorem*{mt1}{Theorem \ref{MFPC}}
\newtheorem*{mt2}{Theorem \ref{QHI}}
\newtheorem*{mt3}{Theorem \ref{Generating}}

\newtheorem{lemma}[theorem]{Lemma}

\newtheorem{proposition}[theorem]{Proposition}
\newtheorem{corollary}[theorem]{Corollary}

\theoremstyle{remark}  
\newtheorem{remark}[theorem]{Remark}
\newtheorem{result}{Result}

\newtheorem{example}[theorem]{Example}
\newtheorem{definition}[theorem]{Definition}

\numberwithin{equation}{section}

\DeclareMathOperator{\ps}{ps}

\title{Quadratic Pseudostable Hodge Integrals and Mumford's Relation}

\author{Renzo Cavalieri}
\address[R. Cavalieri]{Colorado State University\\
Fort Collins, Colorado 80523-1874\\ USA}
\email{renzo@colostate.edu}

\author{Matthew M. Williams}
\address[M. Williams]{Colorado State University\\
Fort Collins, Colorado 80523-1874\\ USA}
\email{matthewmwilliamsmath@gmail.com}

\begin{document}

\begin{abstract} This paper studies the relationship between quadratic Hodge classes on moduli spaces of pseudostable and stable curves given by the contraction morphism $\mathcal{T}.$ While Mumford's relation does not hold in the pseudostable case, we show that one can express the (pullback via $\mathcal{T}$ of the) Chern classes of $\mathbb{E}\oplus \mathbb{E}^\vee$ solely in terms of descendants and strata classes. We organize the combinatorial structure of the pullback of products of two pseudostable $\lambda$ classes and obtain an explicit comparison of arbitrary pseudostable and stable quadratic Hodge integrals.
\end{abstract}

\maketitle


\section{Introduction}
\subsection{Statement of results}

We study the tautological intersection theory of moduli spaces of pseudostable curves. We focus in particular on {\it quadratic Hodge integrals}, i.e., intersection numbers consisting of two lambda classes plus an arbitrary number of psi classes (descendant insertions). Our results show a remarkably elegant combinatorial relationship between pseudostable quadratic Hodge integrals and their stable counterparts. We briefly introduce the notation necessary to state the results here, and defer to Section \ref{sec:Background} for a more comprehensive presentation of the  background.

We denote by $\mathcal{T}$ the contraction morphism
\begin{align}
        \mathcal{T}: \overline{\mathcal{M}}_{g, n} &\rightarrow \overline{\mathcal{M}}^{\ps}_{g, n},
\end{align}
and by $\mathcal{G}^k$  the gluing morphism 
\begin{align}
        \mathcal{G}^k: \overline{\mathcal{M}}_{g-k, n+k} \times \overline{\mathcal{M}}^{\times k}_{1, 1} &\rightarrow \overline{\mathcal{M}}_{g, n}
\end{align}
whose image is the stratum parameterizing curves with $k$ elliptic tails.
Finally, we introduce the {\it psi classes at half-edges} 
\begin{equation}
 \psi_{\star_i} := \mathcal{G}^k_\ast p_0^\ast(\psi_{n+i})   \text{ and } \psi_{\bullet_i} := \mathcal{G}^k_\ast p_i^\ast(\psi_1),
\end{equation}
where $p_i$ is projection onto the $(i+1)$-th factor of the domain of $\mathcal{G}^k$ (in other words, we count the factors starting from $0$ to distinguish the factor corresponding to the ``main'' component of the curve from those for the elliptic tails).

Our first result is the pseudostable version of Mumford's relation, which expresses the total Chern class of $\mathbb{E}\oplus \mathbb{E}^\vee$ in terms of boundary and descendant classes.

\begin{mt1} 
For all pseudostable indices $(g, n)$, we have $$\mathcal{T}^*\left((1 + \lambda_1 + \dots + \lambda_g)(1 - \lambda_1 + \dots + (-1)^g\lambda_g)\right) = \sum_{i=0}^g\frac{1}{i!}\mathcal{G}_*^i\left(\prod_{j=1}^i(\psi_{\star_j}-\psi_{\bullet_j})\right).$$
\end{mt1}

By looking at various homogeneous terms in the expression above, Theorem \ref{MFPC} gives several relations between quadratic polynomials in lambda classes and strata classes decorated solely by descendants ($\psi$ classes). 
The combinatorics for the pull-back of arbitrary quadratic pseudostable Hodge cycles can be organized in a similar fashion. While the cycle theoretic expressions remain rather complicated, one obtains a significant amount of simplification when multiplying by descendant classes and integrating. This allows for an explicit and simple
comparison
result between pseudostable and stable quadratic Hodge integrals.

\begin{mt2}\label{ResultB1}
    For all $i, j = 0, \dots, g$ and polynomial $F \in \mathbb{C}[a_1, \dots, a_n]$, $$\int_{\overline{\mathcal{M}}^{\ps}_{g, n}} \lambda_{i}\lambda_{j}F(\psi_1, \dots, \psi_n) = \sum_{k=0}^{\min\{i, j\}} \frac{1}{24^kk!}\int_{\overline{\mathcal{M}}_{g-k, n+k}} \lambda_{i-k}\lambda_{j-k}F(\psi_1, \dots, \psi_n).$$
\end{mt2}

Theorem \ref{QHI} exhibits a structural relationship between the set of all stable and pseudostable quadratic Hodge integrals. Such relationship becomes transparent when formulated in terms of appropriate generating functions, which are related by a multiplicative factor---a simple exponential in the genus parameter.

\begin{mt3}
    Let $x,y,z,t_m$  be formal variables. Define
    $$ 
    \mathcal{F}^{ps}(x, y, z, \{t_m\}):=\sum_{i, j, g, n = 0}^\infty x^iy^jz^g \int_{\overline{\mathcal{M}}^{\ps}_{g, n}} \frac{\lambda_{g-i}\lambda_{g-j}}{\prod_{m=1}^n (1 - t_m\psi_m)},
    $$
     and define $\mathcal{F}(x, y, z, \{t_m\})$  similarly for stable Hodge integrals. The two generating functions are related by: 
    $$
    \mathcal{F}^{ps}(x, y, z, \{t_m\}) = e^{z/24}\mathcal{F}(x, y, z, \{t_m\}).
    $$
\end{mt3} 

\subsection{Context and motivation}

The moduli space of genus $g$ curves $\mathcal{M}_g$ is a fundamental object in modern algebraic geometry. Riemann's original insight that one should study the continuous parameters (which he called {\it moduli}) governing the deformations of Riemann surfaces has given rise to rich developments in geometry, as well as to several fruitful connections to other fields of mathematics and science, including combinatorics, representation theory, integrable systems and string theory.

A milestone in the study of intersection theory of moduli spaces of curves is Deligne and Mumford's introduction of the notion of stable curves \cite{DMspace}, which gives rise to a smooth and modular compactification  $\mathcal{M}_{g,n}\subset \overline{\mathcal{M}}_{g,n}$, with a normal crossing boundary parameterizing curves with nodal singularities.

The notion of tautological classes (\cite{ac:cag, morita84, Mumford1983}), i.e., Chow (or cohomology) classes defined via the intrinsic geometry of the objects parameterized, identified a rich, and yet accessible slice of the intersection ring of $\overline{\mathcal{M}}_{g,n}$. A set of additive generators for the tautological ring, indexed by decorated graphs is introduced in \cite{graberpandharipande}, and their product structure is spelled out via a combinatorial algorithm in \cite{graberpandharipande, Yang,yang2010intersection}.  
The Chern classes $\lambda_j$ of the Hodge bundle $\mathbb{E}$, and $\psi_i$ of the cotangent line bundles $\mathbb{L}_i$ are tautological, since the above bundles are constructed from the relative dualizing sheaf $\omega_\pi$ via natural operations of push-forward and pull-back.  {\it Hodge integrals}, referring to intersection numbers of arbitrary polynomials in lambda and psi classes, have been studied in \cite{FINC,faber2000hodge, faber2000logarithmic}; the degree of a Hodge integral typically refers to the degree of the polynomial integrated,  where lambda classes are considered variables of degree one, and psi classes of degree zero:  linear Hodge integrals thus consist of intersection numbers with at most one lambda class but an arbitrary number of psi insertions. This is a somewhat unfortunate convention, as lambda and psi classes have also their notion of degree as Chow classes, but it is a rooted tradition in the field.

Through the technique of virtual localization \cite{gp:virtloc}, Gromov-Witten invariants of  targets with a torus action are computed in terms of Hodge integrals; in particular,  quadratic Hodge integrals control the Gromov-Witten theory of any toric surfaces. This has provided a fruitful link connecting tautological intersection theory on the moduli spaces of curves with mirror symmetry (\cite{ck:mirror,clay:mirr}).

Mumford's relation, which refers to the fact that the positive Chern classes of $\mathbb{E}\oplus \mathbb{E}^\vee$ vanish, has served as an essential tool in mirror symmetry, as it allows to reduce the computations of certain families of Gromov-Witten invariants of threefolds, which in principle are cubic Hodge integrals,  in terms of linear Hodge integrals, which are much better understood. This perspective gives a proof of the Aspinwall-Morrison formula \cite{am:top, ele:am} for the computation of Gromov-Witten invariants of the resolved conifold.

In \cite{schubert1991new}, Schubert studies an alternative compactification of the moduli space of curves, where the stability condition is tweaked to allow curves with cuspidal singularities, and to disallow curves with elliptic tails. Such curves are called {\it pseudostable}. The resulting moduli spaces are smooth and proper and have a boundary structure which is not quite normal crossing, but close enough to play well with tropicalization, see \cite{CM:pseudo}. It was later shown \cite{HH} that moduli spaces of pseudostable curves arise in the log minimal model program of $\overline{\mathcal{M}}_{g,n}$, as the birational model following the first wall-crossing.

The study of the tautological intersection theory of moduli spaces of pseudostable curves is initiated in \cite{PSHI}; the strategy  is to
consider the contraction morphism $\mathcal{T}: \overline{\mathcal{M}}_{g, n} \rightarrow \overline{\mathcal{M}}^{\ps}_{g, n}$ and
apply the projection formula to relate tautological cycles on moduli spaces of pseudostable curves to analogous ones on moduli spaces of stable curves. 
It is shown that psi classes are stable under pullback via $\mathcal{T}$, and a combinatorial formula expressing the pullback of lambda classes is given, see Theorem \ref{Pullback}:
\begin{equation}\label{eq:pbla}
 \mathcal{T}^*(\lambda_j) = \sum_{i=0}^j \frac{1}{i!}\mathcal{G}_*^i(p_0^*(\lambda_{j-i})).   
\end{equation}
It is a simple consequence  that stable and pseudostable linear Hodge integrals coincide.

Explicit computations in low genus show that Mumford's relation does not hold in the pseudostable case, and pseudostable quadratic Hodge integrals do not agree with their stable counterparts. While the comparison formulas for individual lambda classes from \cite{PSHI} provide in principle a solution to computing any pseudostable Hodge integral in term of stable ones, the combinatorial complexity of a brute force approach makes it not viable, except for very small values of $g$ and $n$, even with the use of specialized computer software.

The main technical thrust of this work has been to harness the combinatorial complexity of the problem by identifying  non-trivial recursive structure.
In order to compute the total Chern class of $\mathcal{T}^\ast(\mathbb{E}\oplus \mathbb{E}^\vee)$ in genus $g$, we first substitute the relations \eqref{eq:pbla} for each lambda class, then expand the product formally in the decorated graph algebra via the algorithm of \cite{Yang}. The resulting terms can be then partitioned into three disjoint subsets, each of which may be formally bijected with the terms for the entire computation in genus $g-1$. After applying the inductive hypothesis, it then becomes manageable to sum the three resulting terms and obtain the pseudostable version of Mumford's relation.

The remarkably simple structure relating quadratic stable and pseudostable Hodge integrals does not carry over to cubic or higher degree  Hodge integrals. Nonetheless, the availability of natural combinatorial expressions for some quadratic Hodge parts in terms of boundary and descendants opens up the possibility of approaching the study of specific families of cubic Hodge integrals, by using such relations to reduce the degree. 

In particular, we think it should be possible to define notions of {\it pseudostable Gromov-Witten} invariants, and relate them to their stable counterparts via localization. While a careful study of moduli spaces of pseudostable maps and their virtual classes has still not happened, it is reasonable to expect and almost ``guess" what the ingredients contributing to the localized virtual class should be. With this foundational work in place, the pseudostable Gromov-Witten invariants of the resolved conifold should be within reach. It would be interesting to see whether pseudostable Gromov-Witten invariants play a role in the picture of global mirror symmetry, similarly to how other variations of stability conditions do.

\subsection{Structure of the paper} The paper is organized as follows. In Section \ref{sec:Background} we give the background necessary to make this work reasonably self-contained: we introduce moduli spaces of stable and pseudostable curves and the combinatorics of the decorated graph algebra controlling their tautological intersection theory.
In Section \ref{sec:Mumford}, we state and prove the pseudostable version of Mumford's relation, and show as an application some natural relations between Hodge and descendant classes that arise from it.
In Section \ref{sec:quadratic}, we extract the relationship among stable and pseudostable quadratic Hodge integrals, and as an application compute certain families of pseudostable quadratic Hodge integrals.
\subsection{Acknowlegments}
We are grateful to Enrico Arbarello, Seth Ireland, and Dusty Ross for conversations related to the project.
Both authors acknowledge with gratitude support by the NSF, through DMS-2100962.

\section{Background}\label{sec:Background}

\subsection{The moduli spaces of stable and pseudostable curves} Stable and pseudostable curves can have certain types of singularities, and the two of interest are \textit{nodes} and \textit{cusps}. A node is a point locally analytically isomorphic to $x^2 = y^2$, while cusps are locally isomorphic to $x^2 = y^3$. \textit{Special points} of  a curve are any singularities or marked points. On the normalization of a singular curve, inverse images of singularities are also considered special points. 

The moduli spaces of stable and pseudostable curves, denoted $\overline{\mathcal{M}}_{g, n}$ and $\overline{\mathcal{M}}^{\ps}_{g, n}$, respectively, both parameterize families of curves of arithmetic genus $g$ with $n$ marked points. What distinguishes the two spaces are their notions of stability, which we now define and discuss.

\begin{definition}\label{DEFstable} A curve of genus $g$ with $n$ marked points is \emph{stable} if it satisfies the following conditions:
\begin{enumerate}
    \item All marked points are distinct and contained in the smooth locus of the curve.
    \item All singularities of the curve are nodes.
    \item Each connected component of the normalization of the curve with genus $g$ and $s$ special points has $(g, s)$ not equal to $(0, 0), (0, 1), (0, 2)$, or $(1, 0).$
\end{enumerate}
The third condition is equivalent to $g$ and $s$ satisfying $2g - 2 + s > 0$.
\end{definition}

\begin{definition}\label{DEFps} A curve of genus $g$ with $n$ marked points is \emph{pseudostable} if it satisfies the following conditions:
\begin{enumerate}
    \item All marked points are distinct and contained in the smooth locus of the curve.
    \item All singularities of the curve are nodes or cusps.
    \item Each connected component of the normalization of the curve with genus $g$ and $s$ special points has $(g, s)$ not equal to $(0, 0), (0, 1), (0, 2), (1, 0), (1, 1)$, or $(2, 0)$.
\end{enumerate}
The third condition is equivalent to $g$ and $s$ satisfying $g + s > 2$.
\end{definition}
Comparing the definitions shows that the two types of curves are largely the same, but with two key differences. One is that pseudostable curves can have cusps, while stable curves cannot. The other difference is that stable curves can have connected components with indices $(g, s)$ equal to $(1, 1)$ or $(2, 0)$, while pseudostable curves cannot. Components with $(g, s) = (1, 1)$ are referred to as \textit{elliptic tails}.

There is a natural map from $\overline{\mathcal{M}}_{g, n}$ to $\overline{\mathcal{M}}^{\ps}_{g, n}$, where if a stable curve in $\overline{\mathcal{M}}_{g, n}$ has an elliptic tail, that elliptic tail is contracted to a cusp, thus giving a curve in $\overline{\mathcal{M}}^{\ps}_{g, n}$. We define this map more formally in the following definition.

\begin{definition} [As in \cite{PSHI}] \label{Tdef}
    Let $C$ be a stable curve of genus $g$ with marked points $q_1, \dots, q_n$ so that $[(C, q_1, \dots, q_n)] \in \overline{\mathcal{M}}_{g, n}$. Write $$C = C_0 \cup E_1 \cup \dots \cup E_r,$$ where $E_1, \dots, E_r$ are the elliptic tails of $C$. Let $\hat{C}$ be the unique pseudostable curve with marked points $\hat{q}_1, \dots, \hat{q}_n$ that admits a morphism $T: C \rightarrow \hat{C}$ such that 
    \begin{enumerate}
        \item $T$ is an isomorphism when restricted to $C \setminus (E_1 \cup \dots \cup E_r)$,
        \item $T(E_i)$ is a cusp on $\hat{C}$ for all $i = 1, \dots, r$, and
        \item $T(q_i) = \hat{q}_i$.
    \end{enumerate}
    With this notation, we have the following map on moduli spaces:
    \begin{align*}
        \mathcal{T}: \overline{\mathcal{M}}_{g, n} &\rightarrow \overline{\mathcal{M}}^{\ps}_{g, n}\\
        [(C, q_1, \dots, q_n)] &\mapsto [(\hat{C}, \hat{q}_1, \dots, \hat{q}_n)]
    \end{align*}
\end{definition}
Since the two moduli spaces are smooth Deligne-Mumford stacks, the map in Definition \ref{Tdef} induces a pull-back on the Chow rings of the moduli spaces:
$$
\mathcal{T}^\ast: A^\ast(\overline{\mathcal{M}}^{\ps}_{g, n}) \rightarrow A^\ast(\overline{\mathcal{M}}_{g, n}).
$$
\subsection{Tautological classes}

We introduce two types of tautolgical classes which are the main characters in this work: the $\psi$ classes are the first Chern class of
{\it cotangent line bundles}, the $\lambda$ classes are Chern classes of the {\it Hodge bundle}. We adopt a common shortcut in moduli theory and informally introduce these bundles by describing a natural identification of their fibers with a vector space depending on the geometry of the objects parameterized.


The fiber of the $i$-th cotangent line bundle $\mathbb{L}_i$ over a point $[(C, q_1, \dots, q_n)] \in \overline{\mathcal{M}}_{g, n}$ is the cotangent line to $C$ at $q_i$, and $\psi_i$ is defined as the first Chern class of $\mathbb{L}_i$:
$$
\psi_i = c_1(\mathbb{L}_i) \in A^1(\overline{\mathcal{M}}_{g, n}) \text{ for } i = 1, \dots, n.
$$
The fiber of the Hodge bundle $\mathbb{E}$ over a point $[(C, q_1, \dots, q_n)]$ is the $g$-dimensional vector space of global sections of the dualizing sheaf of $C$, and $\lambda_j$ is the $j$-th Chern class of $\mathbb{E}$, that is,
$$
\lambda_j = c_j(\mathbb{E}) \in A^j(\overline{\mathcal{M}}_{g, n}) \text{ for } j = 0, \dots, g.
$$

Both the cotangent line bundle and Hodge bundle can be defined on $\overline{\mathcal{M}}^{\ps}_{g, n}$ in the same way (see \cite[Section 2.2, footnote 1]{PSHI} for an argument of why the pushforward of the relative dualizing sheaf is a vector bundle.), leading to the same definitions of both $\psi$ and $\lambda$ classes: 
$$
\psi_i = c_1(\mathbb{L}_i) \in A^1(\overline{\mathcal{M}}^{\ps}_{g, n}) \text{ for } i = 1, \dots, n,
$$ 
$$
\lambda_j = c_j(\mathbb{E}) \in A^j(\overline{\mathcal{M}}^{\ps}_{g, n}) \text{ for } j = 0, \dots, g.
$$ 

For convenience later on, we introduce the following notation for linear combinations of $\lambda$ classes:
\begin{equation}\label{BigLambda}
    \Lambda_g(k) := \sum_{i=0}^g k^i\lambda_{i}.
\end{equation}

The classes $\Lambda_g(k)$ split multiplicatively on boundary strata. If $G$ is a dual graph of a stable genus $g$, $n$ pointed curve, and $\Delta_G$ the corresponding stratum, then:
\begin{equation}\label{eq:splitLambda}
  \Lambda_g(k) \cdot [\Delta_G] = \prod_{v\in V(G)}  \Lambda_{g_v}(k). 
\end{equation}
An important property of $\lambda$ classes is that they satisfy a collection of relations known as {\it Mumford's relations}. Using notation \eqref{BigLambda}, all these relations are compactly encoded by the formula:
\begin{equation}\label{mumfrel}
 \Lambda_g(1)\Lambda_g(-1) = 1.   
\end{equation}

One way of studying pseudostable $\psi$ and $\lambda$ classes is by looking at their pullbacks $\mathcal{T}^*(\psi_i)$ and $\mathcal{T}^*(\lambda_j)$, where $\mathcal{T}$ is the map described in Definition \ref{Tdef}. The authors in \cite{PSHI} computed explicit formulas for each of these pullbacks, and this result (shown later in Theorem \ref{Pullback}) utilizes the following notation:

\begin{definition}\label{def:gluing}
    For $k \geq 1$, let
    \begin{equation}\label{eq:gluing}
        \mathcal{G}^k: \overline{\mathcal{M}}_{g-k, n+k} \times \overline{\mathcal{M}}^{\times k}_{1, 1} \rightarrow \overline{\mathcal{M}}_{g, n}
    \end{equation}
be the gluing map that sends $(C_0, E_1, \dots, E_k)$ to the curve which attaches each of $E_1, \dots, E_k$ to $C_0$ at the marked points $n+1, \dots, n+k$. If $k=0$, then $\mathcal{G}^0$ is the identity map from $\overline{\mathcal{M}}_{g, n}$ to itself.

Let $p_0$ be projection onto the factor $\overline{\mathcal{M}}_{g-k, n+k}$ of the domain in (\ref{eq:gluing}), and call this factor the \emph{root vertex} of the domain. We denote by $p_i$ the projection onto the $i$-th factor of $\overline{\mathcal{M}}^{\times k}_{1, 1}$ in the domain of (\ref{eq:gluing}).
\end{definition}

The image of the gluing map in (\ref{eq:gluing}) is a stratum parameterizing curves with $k$ elliptic tails. Note that the tautological $\lambda$ and $\psi$ classes may be pulled back from the factors of the product and pushed forward via $\mathcal{G}^k$ to obtain tautological classes supported on this stratum. 

We define 
\begin{equation}\label{eq:starbullet}
    \psi_{\bullet_i} := p_i^\ast(\psi_1) \text{ and } \psi_{\star_i} := p_0^\ast(\psi_{n+i}),
\end{equation}
so that, for example,
$$
\mathcal{G}^1_\ast(-\psi_{\bullet_1} - \psi_{\star_1}) = \mathcal{G}^1_\ast(-p_1^\ast(\psi_1) - p_0^\ast(\psi_{n+1})).
$$

\begin{theorem} [Theorem 2.4, \cite{PSHI}] \label{Pullback} For all pseudostable indices $(g, n)$, we have 
$$
\mathcal{T}^*(\psi_i) = \psi_i \text{ for all } i = 1, \dots, n
$$ 
and 
$$
\mathcal{T}^*(\lambda_j) = \sum_{i=0}^j \frac{1}{i!}\mathcal{G}_*^i(p_0^*(\lambda_{j-i})).
$$
\end{theorem}


For notational convenience, we define 
\begin{equation}\label{eq:lambdahat}
    \hat{\lambda}_j := \mathcal{T}^*(\lambda_j).
\end{equation}
\subsection{Dual graphs of boundary strata} In this subsection, we introduce notation and procedures for performing tautological intersections via an algebra of \emph{decorated dual graphs}; we follow the treatment  in \cite{graberpandharipande, Yang}. 

Each boundary stratum  in $\overline{\mathcal{M}}_{g, n}$ is indexed by the dual graph of the general curve parameterized by points of the stratum; vertices $V$ represent connected components, edges $E$ represent nodes, and half-edges $H$ are marked points of the normalization of the curve. To each vertex, we  associate a number $g \in \mathbb{Z}_{\geq0}$ for the genus of that connected component. There is an involution $i: H \rightarrow H$ which exchanges half-edges that are part of the same edge and has as fixed points the marked point set $N$. We denote by $\Sigma^\ast(\mathcal{\overline{M}}_{g, n})$ the set of dual graphs of stable curves of genus $g$ with $n$ marked points.

\begin{definition}[\cite{graberpandharipande, Yang}]\label{G-structure} Given two dual graphs $G,A \in \Sigma^\ast(\mathcal{\overline{M}}_{g, n})$, a $G$\emph{-structure} on $A$ is an identification of a specialization of $G$ with $A$. Precisely, it is a triple $$(\alpha: V_A \twoheadrightarrow V_G, \beta: H_G \hookrightarrow H_A, \gamma: H_A \setminus \text{im}(\beta) \rightarrow V_G)$$
which satisfies:
\begin{enumerate}
    \item the map $\beta$ commutes with the involution $(\beta \circ i_G = i_A \circ \beta)$ and induces an isomorphism from the fixed points of $G$ to the fixed points of $A$ ($N_G \xrightarrow{\sim} N_A$);
    \item any half-edge $h \in \text{im}(\beta)$ is incident to $v$ if and only if $\beta^{-1}(h)$ is incident to $\alpha(v)$;
    \item if $h \in H_A \setminus \text{im}(B)$ is incident to $v$, then $\gamma(h) = \alpha(v)$;
    \item if $v \in V_G$, then the preimage $(\alpha^{-1}(v), \gamma^{-1}(v)/i_A)$ is a connected graph of genus $g(v)$.
\end{enumerate}
\end{definition}
We find it useful to think of $\beta$ as identifying directed edges of $G$ with their images, and we will informally talk about identified edges to mean the datum $\beta_{|e}: e\to \beta(e)$.

\begin{example}\label{ex:GStructure}
    Consider $\mathcal{G}_*^1(1)$ and $\mathcal{G}_*^2(1)$ in $\overline{\mathcal{M}}_{3, 0}$, which have the  dual graphs $G$ and $A$ depicted in Figure \ref{fig:dualG1G2}.
\begin{figure}[tb]
    \centering
\begin{center}
\begin{tikzpicture}[shorten >=1pt,->]
  \tikzstyle{vertex}=[circle,fill=black!25,minimum size=12pt,inner sep=2pt]
  \node[vertex] (G_2) at (0,0)   {2};
  \node[vertex] (G_3) at (0,1)  {1};
  \node[vertex] (G_4) at (3,1) {1};
  \node[vertex] (G_5) at (4,0)   {1};
  \node[vertex] (G_6) at (5,1)  {1};
   \node at (-1,0.5)  {{$G = $}};
      \node at (2,0.5)  {{$A = $}};
  \draw (G_2) -- (G_3) -- cycle;
  \draw (G_4) -- (G_5) -- (G_6) -- cycle;
\end{tikzpicture}
\end{center}
    \caption{The dual graphs for the strata $\mathcal{G}_*^1(1)$ and $\mathcal{G}_*^2(1)$ in $\overline{\mathcal{M}}_{3, 0}$. }
    \label{fig:dualG1G2}
\end{figure}
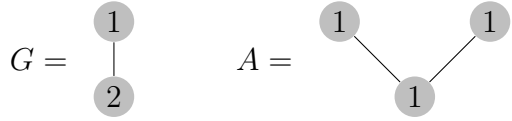
There are two $G$-structures $(\alpha, \beta, \gamma)$ on $A$, one for each of the ways $\beta$ identifies the edge of $G$ with one of the edges of $A$. Once $\beta$ is chosen, the maps $\alpha$ and $\gamma$ are determined based on the conditions from Definition \ref{G-structure}.
\end{example}

If a graph $A$ has both a $G$- and $H$-structure, then $A$ is said to be a $(G, H)$\emph{-graph} with a $(G, H)$\emph{-structure}. A $(G, H)$-graph is \emph{generic} if all its edges are identified with either an edge of $G$, and edge of $H$, or both. The set of all generic $(G, H)$-structures is denoted by $\Gamma(G, H)$. If two $(G, H)$-structures differ by an automorphism of $A$, then they are considered \emph{isomorphic}. Lastly, the \emph{common} $(G, H)$-\emph{edges} of $A$ are those that are identified with both an edge of $G$ and an edge of of $H$.

\begin{example}
    Let $G$ and $A$ again be the graphs shown in Figure \ref{fig:dualG1G2}, and let $H = G$. As shown in Example \ref{ex:GStructure}, there are two $G$-structures and two $H$-structures, so there are four $(G, H)$-structures on $A$. For two of these four, one of the edges of $A$ is identified with (the edge of) $G$, and the other is identified with $H$, and these are the two generic $(G, H)$-structures of $A$. Up to isomorphism, there is only one generic $(G, H)$-structure since the two differ by an automorphism of $A$. For this generic $(G, H)$-graph, there are no common $(G, H)$-edges.
\end{example}


For strata supporting $\lambda$ and $\psi$ classes, such as those seen in Theorem \ref{Pullback}, the classes are represented as decorations of the dual graph $G$: $\psi$ classes on half-edges of $G$, and $\lambda$ classes on vertices of $G$.
We denote by $\theta_v$ the decorations for a given vertex $v$: \begin{equation}\label{eq:vertdec}
  \theta_v = \prod_{i=1}^{n_v}\psi_i^{e_i}\prod_{j=1}^{f}\Lambda_{g_v}(k_j).  
\end{equation} 
While the shape of the Hodge part of the decoration in \eqref{eq:vertdec} appears very specific, one can in fact obtain an arbitrary monomial in $\lambda$ classes as the coefficient of the corresponding monomial in the variables $k_j$. Graphs with these decorations are called \emph{decorated graphs}. We denote $\underline{G}$ to be the underlying graph of $G$, that is, the graph $G$ without any decorations.

\begin{remark}
There is a natural geometric reason why two $\psi$ classes decorating two half-edges that are an orbit of the involution (i.e., $(h_-,h_+)$ belonging to the same edge $e$) often appear summed together: for a divisor of compact type, this sum is the Euler class of its normal bundle. While technically this sum should correspond to a sum of two graphs, each with one of the half-edges decorated by the $\psi$ class, it is extremely convenient to abuse notation and say that this sum correspond to a single graph with an edge decorated by $(\psi_{h_-} + \psi_{h_+})$. We adopt this convention, extending it by bilinearity to arbitrary polynomial decorations of $e$ by the classes $\psi_{h_-}, \psi_{h_+}$.   
\end{remark}

\begin{definition}[based on \cite{graberpandharipande, Yang}]\label{Yang} Let $G$ and $H$ denote two decorated graphs with vertex decorations as in \eqref{eq:vertdec}, and let $\Gamma(G, H)$ denote the set of all generic $(G, H)$-structures $(\alpha, \beta, \gamma)$. The product of  $G$ and $H$ is given by:
$$
G \cdot H := \sum_{A\in\Gamma(\underline{G}, \underline{H})} \frac{1}{|\text{Aut}(A)|}F_A(G, H)\cdot A,
$$
where 
\begin{align*}
    f_A(G, v) &= \prod_{i=1}^{n_v}\psi_{\beta(i)}^{e_i}\prod_{j=1}^{f}\left(\prod_{w \in \alpha^{-1}(v)}\Lambda_{g_w}(k_j)\right),\\
    F_A(G, H) &= \prod_{v \in V_G} f_A(G, v) f_A(H, v) \prod_{e_i = (\bullet_i, \star_i)} (-\psi_{\bullet_i} - \psi_{\star_i}).
\end{align*}
The last product above is taken over all common $(\underline{G},\underline{H})$-edges of $A$, and $-\psi_{\bullet_i}$ and $- \psi_{\star_i}$ are psi classes associated to half-edges of the dual graph, defined analogously as in (\ref{eq:starbullet}).
\end{definition} 

Yang shows in \cite{Yang} that this operation turns the set of decorated graphs into an algebra, and that there's a ring homomorphism from the algebra of decorated dual graphs to the tautological ring of $\overline{\mathcal{M}}_{g,n}$, namely 
\begin{equation}\label{eq: Yang}
    \Phi : \Sigma^\ast(\overline{\mathcal{M}}_{g, n}) \rightarrow R^\ast(\overline{\mathcal{M}}_{g, n}), 
\end{equation}
which sends a decorated dual graph to its associated tautological class.

\begin{example}
Consider $\mathcal{G}_*^1(1)$ and $\mathcal{G}_*^2(1)$ in $\overline{\mathcal{M}}_{3, 0}$, with dual graphs $G$ and $H$, respectively, shown in Figure \ref{fig:2dualG1G2}. Up to isomorphism, there are two generic $(\underline{G}, \underline{H})$-graphs, which we'll denote $A_1$ and $A_2,$ depicted in Figure \ref{fig:GenericA1A2}.

Note that $|\text{Aut}(A_1)| = 2!$, but for each common edge of $G$ and $H$ in $A_1$, there are $\binom{2}{1} = 2$ ways to choose an edge from $H$, so the coefficient of these terms is $\frac{2}{2!} = 1.$ Similarly, $|\text{Aut}(A_2)| = 3!,$ but there are $3!$ ways to order the $3$ collective edges of $G$ and $H$, so the coefficient of $A_2$ is $\frac{3!}{3!} = 1.$ 
\begin{align*}
    G \cdot H &= \frac{2}{2!}(-\psi_{\bullet_1} - \psi_{\star_1})A_1 + \frac{2}{2!}(-\psi_{\bullet_2} - \psi_{\star_2})A_1 + \frac{3!}{3!}A_2\\
    &= (-\psi_{\bullet_1} - \psi_{\star_1})A_1 + (-\psi_{\bullet_2} - \psi_{\star_2})A_1 + A_2
\end{align*}
This result is depicted in Figure \ref{fig:Product}.

\begin{figure}[tb]
    \centering
\begin{center}
\begin{tikzpicture}[shorten >=1pt,->]
  \tikzstyle{vertex}=[circle,fill=black!25,minimum size=12pt,inner sep=2pt]
  \node[vertex] (G_2) at (0,0)   {2};
  \node[vertex] (G_3) at (0,1)  {1};
  \node[vertex] (G_4) at (3,1) {1};
  \node[vertex] (G_5) at (4,0)   {1};
  \node[vertex] (G_6) at (5,1)  {1};
   \node at (-1,0.5)  {{$G = $}};
      \node at (2,0.5)  {{$H = $}};
  \draw (G_2) -- (G_3) -- cycle;
  \draw (G_4) -- (G_5) -- (G_6) -- cycle;
\end{tikzpicture}
\end{center}
    \caption{The dual graphs for the strata $\mathcal{G}_*^1(1)$ and $\mathcal{G}_*^2(1)$ in $\overline{\mathcal{M}}_{3, 0}$. }
    \label{fig:2dualG1G2}
\end{figure}
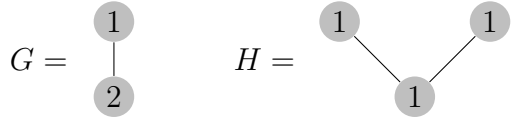

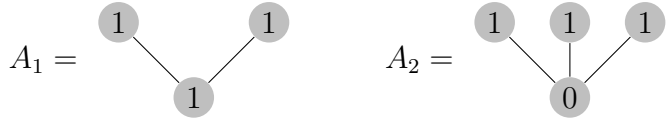
\begin{figure}[tb]
    \centering
\begin{center}
\begin{tikzpicture}[shorten >=1pt,->]
  \tikzstyle{vertex}=[circle,fill=black!25,minimum size=12pt,inner sep=2pt]
  \node[vertex] (G_1) at (0,0)   {1};
  \node[vertex] (G_2) at (-1,1)  {1};
  \node[vertex] (G_3) at (1,1) {1};
  \node[vertex] (G_4) at (4,1) {1};
  \node[vertex] (G_5) at (5,0)   {0};
  \node[vertex] (G_6) at (5,1)  {1};
  \node[vertex] (G_7) at (6,1)  {1};
     \node at (-2,0.5)  {{$A_1 = $}};
     \node at (3,0.5)  {$A_2 = $};
  \draw (G_2) -- (G_1) -- (G_3) -- cycle;
  \draw (G_4) -- (G_5) -- (G_6) -- cycle;
  \draw (G_5) -- (G_7) -- cycle;
\end{tikzpicture}
\end{center}
    \caption{Up to isomorphism, the two generic $(G, H)$-graphs for $G$ and $H$ in Figure \ref{fig:2dualG1G2}.}
    \label{fig:GenericA1A2}
\end{figure}

\begin{figure}[tb]
    \centering
\begin{center}
\begin{tikzpicture}[shorten >=1pt,->]
  \tikzstyle{vertex}=[circle,fill=black!25,minimum size=12pt,inner sep=2pt]
  \node[vertex] (G_1) at (0,0)   {1};
  \node[vertex] (G_2) at (-1,1)  {1};
  \node[vertex] (G_3) at (1,1) {1};
  \node[vertex] (G_4) at (3.5,0)   {1};
  \node[vertex] (G_5) at (2.5,1)  {1};
  \node[vertex] (G_6) at (4.5,1) {1};
  \node[vertex] (G_7) at (7,0)   {1};
  \node[vertex] (G_8) at (6,1)  {1};
  \node[vertex] (G_9) at (8,1) {1};
     \node at (-.95,0.55)  {$-\psi_{\bullet_1}$};
     \node at (2.85,0.25)  {$-\psi_{\star_1}$};
     \node at (8,0.55)  {$-\psi_{\bullet_2}$};
     \node at (1.75,0.5)  {{$+$}};
     \node at (5.25,0.5)  {{$+$}};
  \draw (G_2) -- (G_1) -- (G_3) -- cycle;
  \draw (G_5) -- (G_4) -- (G_6) -- cycle;
  \draw (G_8) -- (G_7) -- (G_9) -- cycle;
\end{tikzpicture}
\begin{tikzpicture}[shorten >=1pt,->]
  \tikzstyle{vertex}=[circle,fill=black!25,minimum size=12pt,inner sep=2pt]
  \node[vertex] (G_1) at (-0.25,0)   {1};
  \node[vertex] (G_2) at (-1.25,1)  {1};
  \node[vertex] (G_3) at (0.75,1) {1};
  \node[vertex] (G_4) at (2.25,1) {1};
  \node[vertex] (G_5) at (3.25,0)   {0};
  \node[vertex] (G_6) at (3.25,1)  {1};
  \node[vertex] (G_7) at (4.25,1)  {1};
     \node at (0.45,0.25)  {$-\psi_{\star_2}$};
     \node at (-2,0.5)  {{$+$}};
     \node at (1.5,0.5)  {{$+$}};
  \draw (G_2) -- (G_1) -- (G_3) -- cycle;
  \draw (G_4) -- (G_5) -- (G_6) -- cycle;
  \draw (G_5) -- (G_7) -- cycle;
\end{tikzpicture}
\end{center}
    \caption{The resulting product of $G$ and $H$.}
    \label{fig:Product}
\end{figure}
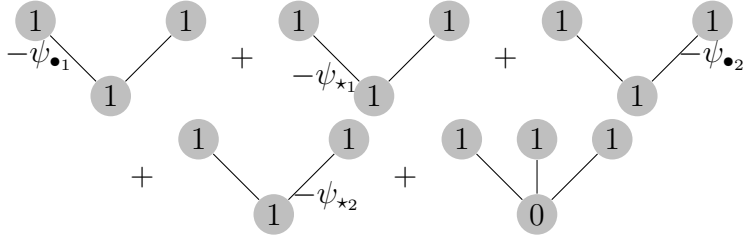

\end{example}

While this fact is not strictly relevant to this work, we mention for the benefit of the reader that this algorithm has been coded in the freely available program {\tt admcycles} \cite{delecroix2022admcycles}, which is a powerful tool for making and checking computations in the tautological ring.

\section{Mumford's relation for pseudostable $\lambda$ classes}\label{sec:Mumford}

In this section, we prove the pseudostable version of Mumford's relation (Theorem \ref{MFPC}) and provide some applications to the computation of Hodge integrals.

\begin{theorem} \label{MFPC} For all pseudostable indices $(g, n)$, we have 
$$
(1 + \hat\lambda_1 + \dots + \hat\lambda_g)(1 - \hat\lambda_1 + \dots + (-1)^g\hat\lambda_g) =  \sum_{i=0}^g\frac{1}{i!}\mathcal{G}_*^i\left(\prod_{j=1}^i(\psi_{\star_j}-\psi_{\bullet_j})\right).
$$
\end{theorem}

The proof of Theorem \ref{MFPC} is rather technical, so before we start, we study toy examples that will help in understanding the general strategy. The proof of Theorem \ref{MFPC} will be by induction on $g$, and the following example  serves as the base case with $g = 1$.

\begin{example}\label{g=1} For genus $g=1$, and any fixed $n\geq 1$,  consider
$ (1 + \hat{\lambda}_1)(1 - \hat{\lambda}_1) \in A^*(\overline{\mathcal{M}}_{1, n})
$, 
where $\hat{\lambda}_1 = \lambda_1 + \mathcal{G}_*^1(1)$ by Theorem \ref{Pullback}. Making this substitution, we obtain 
$$
(1 + \lambda_1 + \mathcal{G}_*^1(1))(1 - \lambda_1 - \mathcal{G}_*^1(1)),
$$
which using the $\Lambda_g(k)$ notation from (\ref{BigLambda}) is equal to
$$
(\Lambda_1(1) + \mathcal{G}_*^1(1))(\Lambda_1(-1) - \mathcal{G}_*^1(1)).
$$
This product produces $4$ terms: 
\begin{enumerate}
    \item $\Lambda_1(1)\Lambda_1(-1) = 1$, being Mumford's relation for stable curves.
    \item $\Lambda_1(1)\cdot -\mathcal{G}_*^1(1) = - \mathcal{G}_*^1(p_1^\ast(\Lambda_1(1)))$
    \item $\mathcal{G}_*^1(1) \cdot \Lambda_1(-1) = \mathcal{G}_*^1(p_1^\ast(\Lambda_1(-1)))$
    \item $\mathcal{G}_*^1(1)\cdot -\mathcal{G}_*^1(1) = -\mathcal{G}_*^1(-\psi_{\bullet_1}-\psi_{\star_1})$
\end{enumerate}
Recall that when  $(g,n) = (1,1)$, $\lambda_1 = \psi_{\bullet_1},$ (\cite[Proposition 2.21  and Example 2.26]{zvonkine2014introduction}) so these four terms sum together and give
\begin{align*}
    &1 - \mathcal{G}_*^1(p_1^\ast(\Lambda_1(1))) + \mathcal{G}_*^1(p_1^\ast(\Lambda_1(-1))) - \mathcal{G}_*^1(-\psi_{\bullet_1}-\psi_{\star_1})\\ 
    = &1 - \mathcal{G}_*^1(p_1^\ast(1 + \lambda_1)) + \mathcal{G}_*^1(p_1^\ast(1 - \lambda_1)) - \mathcal{G}_*^1(-\psi_{\bullet_1}-\psi_{\star_1})\\
    = &1 + \mathcal{G}_*^1((-1 -\psi_{\bullet_1}) + (1 - \psi_{\bullet_1}) + (\psi_{\bullet_1} + \psi_{\star_1}))\\
    = &1 + \mathcal{G}_*^1(\psi_{\star_1}-\psi_{\bullet_1}),
\end{align*}
which shows that Theorem \ref{MFPC} holds in the case of $g=1$.
\end{example}

In the following example, we show a similar computation to the one in Example \ref{g=1}, but in arbitrary genus $g$: we wish to illustrate that if we restrict our attention only to strata parameterizing curves with at most one elliptic tail, then the computation is formally identical independently of the genus. This statement will be true for any fixed number of elliptic tails.

\begin{example}\label{g=g} For genus $g\geq 2$, consider
$$
 (\Lambda_g(1) + \mathcal{G}_*^1(p_0^\ast(\Lambda_{g-1}(1)))) \cdot (\Lambda_g(-1) - \mathcal{G}_*^1(p_0^\ast(\Lambda_{g-1}(-1)))) \in A^*(\overline{\mathcal{M}}_{g, n}).
$$
Distributing the terms gives $4$ products: 
\begin{enumerate}
    \item $$\Lambda_g(1)\Lambda_g(-1) = 1$$
    \item 
    \begin{align*}
        \Lambda_g(1)\cdot -\mathcal{G}_*^1(p_0^\ast(\Lambda_{g-1}(-1))) &= - \mathcal{G}_*^1(p_0^\ast(\Lambda_{g-1}(-1)\Lambda_{g-1}(1))p_1^\ast(\Lambda_1(1)))\\ &= -\mathcal{G}_*^1(p_1^\ast(\Lambda_1(1)))
    \end{align*}
    (Note: The final equality uses Mumford's relation in genus $g-1$.)
    \item 
    \begin{align*}
        \mathcal{G}_*^1(p_0^\ast(\Lambda_{g-1}(1))) \cdot \Lambda_g(-1) &= \mathcal{G}_*^1(p_0^\ast(\Lambda_{g-1}(-1)\Lambda_{g-1}(1))p_1^\ast(\Lambda_1(-1)))\\ &= \mathcal{G}_*^1(p_1^\ast(\Lambda_1(-1)))
    \end{align*}
    
    \item 
    \begin{align*}
    &\mathcal{G}_*^1(p_0^\ast(\Lambda_{g-1}(1)))\cdot -\mathcal{G}_*^1(p_0^\ast(\Lambda_{g-1}(-1)))\\
    = -&\mathcal{G}_*^1(p_0^\ast(\Lambda_{g-1}(-1)\Lambda_{g-1}(1))(-\psi_{\bullet_1}-\psi_{\star_1})) - \mathcal{G}_\ast^2(p_0^\ast(\Lambda_{g}(-1)\Lambda_{g}(1)) \\
    = -&\mathcal{G}_*^1(-\psi_{\bullet_1}-\psi_{\star_1}) - \mathcal{G}_\ast^2(1)
    \end{align*}
\end{enumerate}
Notice the only difference between the results above and those from Example \ref{g=1} is one additional term in the 4th product, that being $- \mathcal{G}_\ast^2(1)$. However, all terms supported on strata parameterizing curves with fewer than $2$ elliptic tails are in bijection with the 4 terms from Example \ref{g=1}, and their sum returns the same answer: $$1 + \mathcal{G}_*^1(\psi_{\star_1}-\psi_{\bullet_1}).$$
We will use this same idea in the proof of Theorem \ref{MFPC} by showing that the computation in genus $g$ is formally identical to that in genus $g-1$ when considering terms supported on strata parameterizing curves with fewer than $g$ elliptic tails.
\end{example}

In the following example, we study the terms of the genus $2$ case of Theorem \ref{MFPC} that are supported on strata parameterizing curves with exactly  $2$ elliptic tails. 

\begin{example}\label{g=2} Consider the following product in $A^*(\overline{\mathcal{M}}_{2, n})$:
    $$
 (\Lambda_2(1) + \mathcal{G}_*^1(p_0^\ast(\Lambda_{1}(1))) + \frac{1}{2!}\mathcal{G}_*^2(1)) \cdot (\Lambda_2(-1) - \mathcal{G}_*^1(p_0^\ast(\Lambda_{1}(-1))) + \frac{1}{2!}\mathcal{G}_*^2(1)).
    $$
The only products which produce terms supported on strata parameterizing curves with $2$ elliptic tails are the following:

\begin{enumerate}
    \item 
    $$
    \Lambda_2(1) \cdot \frac{1}{2!}\mathcal{G}_*^2(1) 
    $$
    $$
    = \frac{1}{2!}\mathcal{G}_*^2(p_1^\ast(\Lambda_{1}(1))p_2^\ast(\Lambda_{1}(1)))
    $$
    
    \item 
    $$
    \mathcal{G}_*^1(p_0^\ast(\Lambda_{1}(1))) \cdot (- \mathcal{G}_*^1(p_0^\ast(\Lambda_{1}(-1))))
    $$
    $$
    = -\mathcal{G}_*^1(-\psi_{\bullet_1}-\psi_{\star_1}) - \frac{1}{2!}\mathcal{G}_*^2(p_1^\ast(\Lambda_{1}(1))p_2^\ast(\Lambda_{1}(-1))) - \frac{1}{2!}\mathcal{G}_*^2(p_1^\ast(\Lambda_{1}(-1))p_2^\ast(\Lambda_{1}(1)))
    $$
    
    \item 
    $$
    \mathcal{G}_*^1(p_0^\ast(\Lambda_{1}(1)))\cdot \frac{1}{2!}\mathcal{G}_*^2(1)
    $$
    $$ = \frac{1}{2!}\mathcal{G}_*^2((-\psi_{\bullet_1}-\psi_{\star_1})p_2^\ast(\Lambda_{1}(1))) + \frac{1}{2!}\mathcal{G}_*^2(p_1^\ast(\Lambda_{1}(1))(-\psi_{\bullet_2}-\psi_{\star_2}))
    $$
    
    \item 
    $$
    \frac{1}{2!}\mathcal{G}_*^2(1) \cdot \Lambda_2(-1)
    $$
    $$
    = \frac{1}{2!}\mathcal{G}_*^2(p_1^\ast(\Lambda_{1}(-1))p_2^\ast(\Lambda_{1}(-1)))
    $$
    
    \item 
    $$
    \frac{1}{2!}\mathcal{G}_*^2(1) \cdot (-\mathcal{G}_*^1(p_0^\ast(\Lambda_{1}(1)))) 
    $$
    $$
    = -\frac{1}{2!}\mathcal{G}_*^2((-\psi_{\bullet_1}-\psi_{\star_1})p_2^\ast(\Lambda_{1}(-1))) - \frac{1}{2!}\mathcal{G}_*^2(p_1^\ast(\Lambda_{1}(-1))(-\psi_{\bullet_2}-\psi_{\star_2}))
    $$
    
    \item 
    $$
    \frac{1}{2!}\mathcal{G}_*^2(1) \cdot \frac{1}{2!}\mathcal{G}_*^2(1)
    $$
    $$
    = \frac{1}{2!}\mathcal{G}_*^2((-\psi_{\bullet_1}-\psi_{\star_1})(-\psi_{\bullet_2}-\psi_{\star_2}))
    $$
\end{enumerate}

From the $6$ products above, there are $9$ terms that are supported on strata parameterizing curves with $2$ elliptic tails. These can be partitioned into groups of three based on a common factor on the first elliptic tail.

Group 1, $\Lambda_1(1)$ on the first elliptic tail: 
$$
\frac{1}{2!}\mathcal{G}_*^2(p_1^\ast(\Lambda_{1}(1))p_2^\ast(\Lambda_{1}(1)))
$$
$$
- \frac{1}{2!}\mathcal{G}_*^2(p_1^\ast(\Lambda_{1}(1))p_2^\ast(\Lambda_{1}(-1))
$$
$$
\frac{1}{2!}\mathcal{G}_*^2(p_1^\ast(\Lambda_{1}(1))(-\psi_{\bullet_2}-\psi_{\star_2}))
$$

Group 2, $\Lambda_1(-1)$ on the first elliptic tail:
$$
- \frac{1}{2!}\mathcal{G}_*^2(p_1^\ast(\Lambda_{1}(-1))p_2^\ast(\Lambda_{1}(1)))
$$
$$
\frac{1}{2!}\mathcal{G}_*^2(p_1^\ast(\Lambda_{1}(-1))p_2^\ast(\Lambda_{1}(-1)))
$$
$$
-\frac{1}{2!}\mathcal{G}_*^2(p_1^\ast(\Lambda_{1}(-1))(-\psi_{\bullet_2}-\psi_{\star_2}))
$$

Group 3, $-\psi_{\bullet_1}-\psi_{\star_1}$ on the first elliptic tail:
$$
\frac{1}{2!}\mathcal{G}_*^2((-\psi_{\bullet_1}-\psi_{\star_1})p_2^\ast(\Lambda_{1}(1))) 
$$
$$
-\frac{1}{2!}\mathcal{G}_*^2((-\psi_{\bullet_1}-\psi_{\star_1})p_2^\ast(\Lambda_{1}(-1)))
$$
$$
\frac{1}{2!}\mathcal{G}_*^2((-\psi_{\bullet_1}-\psi_{\star_1})(-\psi_{\bullet_2}-\psi_{\star_2}))
$$

Notice in each group of three, the terms on the second elliptic tail are in bijection with the three terms added together at the end of Example \ref{g=1}. Thus, the terms in each group sum together in the same way (up to a negative) on the second elliptic tail:
$$
-p_2^\ast(\Lambda_{1}(1)) + p_2^\ast(\Lambda_{1}(-1)) - (-\psi_{\bullet_2}-\psi_{\star_2}) = \psi_{\star_2}-\psi_{\bullet_2}.
$$
Each of the groups are now as follows:

Group 1: 
$$
-\frac{1}{2!}\mathcal{G}_*^2(p_1^\ast(\Lambda_{1}(1))(\psi_{\star_2}-\psi_{\bullet_2}))
$$

Group 2: 
$$
\frac{1}{2!}\mathcal{G}_*^2(p_1^\ast(\Lambda_{1}(-1))(\psi_{\star_2}-\psi_{\bullet_2}))
$$

Group 3: 
$$
-\frac{1}{2!}\mathcal{G}_*^2((-\psi_{\bullet_1}-\psi_{\star_1})(\psi_{\star_2}-\psi_{\bullet_2}))
$$

Now that the computation has been reduced from $9$ terms to $3$, notice that these $3$ terms all have the same decorations on the second elliptic tail, namely $(\psi_{\star_2}-\psi_{\bullet_2})$, and the decorations on the first elliptic tail are in bijection with the same three terms from the end of Example \ref{g=1} as before. Thus, these three terms add together  on the first elliptic tail in the same way:
$$
-p_1^\ast(\Lambda_{1}(1)) + p_1^\ast(\Lambda_{1}(-1)) - (-\psi_{\bullet_1}-\psi_{\star_1}) = \psi_{\star_1}-\psi_{\bullet_1}.
$$
This leads in the following sum of the groups above:
$$
\frac{1}{2!}\mathcal{G}_*^2((\psi_{\star_1}-\psi_{\bullet_1})(\psi_{\star_2}-\psi_{\bullet_2}))
$$
This is exactly the term supported on $2$ elliptic tails claimed in the statement of Theorem \ref{MFPC}. We'll use the same idea in the proof by showing that for the genus $g$ computation,  the terms supported on $g$ elliptic tails  can be partitioned into three groups, each of which is in bijection with the terms supported on on $g-1$ elliptic tails in the genus $g-1$ computation.
\end{example}

We are now ready to prove Theorem \ref{MFPC}.

\begin{proof} We must compute
\begin{align} \label{eq:start}
     (&1 + \hat{\lambda}_1 + \dots + \hat{\lambda}_g)(1 - \hat{\lambda}_1 + \dots + (-1)^g\hat{\lambda}_g),
\end{align}
where by Theorem \ref{Pullback}, $\hat{\lambda}_j = \sum_{i=0}^j\frac{1}{i!}\mathcal{G}_*^i(p_0^*(\lambda_{j-i})).$ Substituting these expressions, exchanging the order of summation, and using the $\Lambda_{g}(k)$ notation, we can rewrite \eqref{eq:start}  as  
\begin{equation} \label{eq:big}
\left(\sum_{i=0}^g\frac{1}{i!}\mathcal{G}_*^i(p_0^*(\Lambda_{g-i}(1)))\right)\left(\sum_{j=0}^g\frac{(-1)^j}{j!}\mathcal{G}_*^j(p_0^*(\Lambda_{g-j}(-1)))\right). 
\end{equation}

We'll proceed by induction on $g.$ For the base case of $g = 1,$ see Example \ref{g=1}. Assume that the theorem is true for genus $g - 1,$ meaning 
\begin{equation}\label{eq:g-1}
    \left(\sum_{i=0}^{g-1}\frac{1}{i!}\mathcal{G}_*^i(p_0^*(\Lambda_{(g-1)-i}(1)))\right)\left(\sum_{j=0}^{g-1}\frac{(-1)^j}{j!}\mathcal{G}_*^j(p_0^*(\Lambda_{(g-1)-j}(-1)))\right)
\end{equation}
is equal to  
\begin{equation}\label{eq:induction}
    \sum_{i=0}^{g-1}\frac{1}{i!}\mathcal{G}_*^i\left(\prod_{j=1}^i(\psi_{\star_j}-\psi_{\bullet_j})\right)
\end{equation}
when working in genus $g-1$. 

In the genus $g$ case, the truncated product with the first $g-1$ terms is instead written as
$$
\left(\sum_{i=0}^{g-1}\frac{1}{i!}\mathcal{G}_*^i(p_0^*(\Lambda_{g-i}(1)))\right)\left(\sum_{j=0}^{g-1}\frac{(-1)^j}{j!}\mathcal{G}_*^j(p_0^*(\Lambda_{g-j}(-1)))\right),
$$ 
with the only difference from (\ref{eq:g-1}) being $\Lambda_{g-i}(1)$ and $\Lambda_{g-j}(-1)$ instead of $\Lambda_{(g-1)-i}(1)$ and $\Lambda_{(g-1)-j}(-1)$, respectively. However, the genus of the factor corresponding to the root vertex also increases by $1$, so the resulting computation will be formally identical to the genus $g-1$ case for all terms supported on the strata parameterizing curves with fewer than $g$ elliptic tails. One can define a bijection between the decorated graphs by assigning to a decorated graph of genus $g$ a corresponding genus $g-1$ decorated graph where the genus of the root vertex has been lowered by one;  this was demonstrated in Example \ref{g=g} by showing that the computation in the genus $g \geq 2$ case is the same as the genus $1$ case for terms supported on strata parameterizing curves with fewer than $2$ elliptic tails. Using the inductive hypothesis, the terms supported on strata parameterizing curves with fewer than $g$ elliptic tails simplify to
$$
\sum_{i=0}^{g-1}\frac{1}{i!}\mathcal{G}_*^i\left(\prod_{j=1}^i(\psi_{\star_j}-\psi_{\bullet_j})\right).
$$

What remains to show is that all terms supported on strata parameterizing curves with exactly $g$ elliptic tails sum to 
$$
\frac{1}{g!}\mathcal{G}_*^g\left(\prod_{j=1}^g(\psi_{\star_j}-\psi_{\bullet_j})\right).
$$

Consider all terms in the product (\ref{eq:big}) which are supported on a stratum parameterizing curves with exactly $g$ elliptic tails. Each term has a dual decorated graph $A$ with $g$ genus $1$ vertices, which has a coefficient based on the $(G, H)$-structure obtained from intersecting two strata $G$ and $H$ with an $G$-structure and $H$-structure, respectively. Each coefficient will include $\frac{1}{|\text{Aut}(A)|} = \frac{1}{g!}$ as in Definition \ref{Yang}.

Let the graph $G$ correspond to a term in the left sum of (\ref{eq:big}) and the graph $H$ correspond to a term in the right sum of (\ref{eq:big}), and consider the last edge and genus one vertex of the graph $A$. If the last edge of $A$ is only from the $G$-structure, then its coefficient will decorate the vertex with  $\Lambda_1(-1)$. If the last edge of $A$ is only from the $H$-structure, then its coefficient will have $-\Lambda_1(1)$ on the genus one vertex. If the last edge of $A$ is from both the $G$- and $H$-structure, then the edge is decorated by $-(-\psi_{\bullet_1} -\psi_{\star_1})$.

Consider the set of all possible graphs such that the coefficient that is pulled back from the last factor (i.e., decorating the last genus one vertex) is $\Lambda_1(-1).$ If we factor out this coefficient, what remains will be those coefficients that were pulled back from the other factors. As in Example \ref{g=2}, removing the last edge from every decorated graph gives a natural bijection between these remaining coefficients and the coefficients of the factors of the term supported on $g-1$ elliptic tails; we have shown inductively in (\ref{eq:induction}) those add up to:
$$
F_{g-1} := \prod_{j=1}^{g-1}(\psi_{\star_j}-\psi_{\bullet_j}),
$$
decorating the first $g-1$ edges of the dual graph. The same argument holds for the collection of all terms where the last vertex is decorated by $-\Lambda_1(1)$ or by $-(-\psi_{\bullet_1} -\psi_{\star_1})$. By induction, summing all these terms, the first $g-1$ edges  will be decorated by the same coefficient of $F_{g-1}$, so we can sum over the last factor to get 
$$
-p_g^\ast(\Lambda_{1}(1)) + p_g^\ast(\Lambda_{1}(-1)) - (-\psi_{\bullet_g}-\psi_{\star_g}) = \psi_{\star_g}-\psi_{\bullet_g}.
$$
Along with $F_{g-1}$ and $\frac{1}{g!},$ this returns the desired coefficient for strata parameterizing curves with exactly $g$ elliptic tails: 
$$\frac{1}{g!}F_{g-1}(\psi_{\star_g}-\psi_{\bullet_g}) = \frac{1}{g!}\prod_{j=1}^{g}(\psi_{\star_j}-\psi_{\bullet_j}),
$$
which concludes the proof.
\end{proof}

\subsection{Applications}

Much like the classic version of Mumford's formula, Theorem \ref{MFPC} provides a set of relations for $\lambda$ classes in $A^*(\overline{\mathcal{M}}^{\ps}_{g, n})$ for all pseudostable indices $(g, n)$. 

\begin{corollary}\label{Relations}
    For all pseudostable indices $(g, n)$ and $i = 0, \dots, g$, we have
    $$
    \sum_{j = \max\{0, 2i-g\}}^{\min\{2i, g\}}(-1)^j\hat{\lambda}_{2i-j}\hat{\lambda}_j = \frac{1}{i!}\mathcal{G}_*^i\left(\prod_{j=1}^i(\psi_{\star_j}-\psi_{\bullet_j})\right).
    $$
\end{corollary}
\begin{proof}
With Theorem \ref{MFPC}, we know that the codimension $2i$ terms of 
\begin{equation}\label{eq:2iMumford}
    (1 + \hat{\lambda}_1 + \dots + \hat{\lambda}_g)(1 - \hat{\lambda}_1 + \dots + (-1)^g\hat{\lambda}_g)
\end{equation}
are equal to the codimension $2i$ term of 
\begin{equation}\label{eq:2igraphs}
    \sum_{i=0}^g\frac{1}{i!}\mathcal{G}_*^i\left(\prod_{j=1}^i(\psi_{\star_j}-\psi_{\bullet_j})\right).
\end{equation}

The codimension $2i$ terms of $(\ref{eq:2iMumford})$ are of the form
$$
(-1)^j\hat{\lambda}_{2i-j}\hat{\lambda}_j,
$$
where $\max\{0, 2i - g\} \leq j \leq \min\{2i, g\}$. The reason for $2i - g$ in the lower bound is that $\hat{\lambda}_{2i-j} = 0$ when $2i - j > g$. Similarly, we include $g$ in the upper bound of $j$ because $\hat{\lambda}_j = 0$ for $j > g$.

Finally, the codimension $2i$ term of (\ref{eq:2igraphs}) is 
$$
\frac{1}{i!}\mathcal{G}_*^i\left(\prod_{j=1}^i(\psi_{\star_j}-\psi_{\bullet_j})\right)
$$
since the codimension of each of $\mathcal{G}_*^i(1)$ and $\prod_{j=1}^i(\psi_{\star_j}-\psi_{\bullet_j})$ is $i$.
\end{proof}

The following example highlights a few specific relations from Corollary \ref{Relations}.

\begin{example}
Codimension $2g$ ($i = g$ in Corollary \ref{Relations}):
$$
\hat{\lambda}_g^2 = \frac{(-1)^g}{g!}\mathcal{G}_*^g\left(\prod_{j=1}^g(\psi_{\star_j} - \psi_{\bullet_j})\right).
$$
Codimension $2g-2$ ($i = g-1$ in Corollary \ref{Relations}):
$$
2\hat{\lambda}_g\hat{\lambda}_{g-2} - \hat{\lambda}_{g-1}^2 = \frac{(-1)^g}{(g-1)!}\mathcal{G}_*^{g-1}\left(\prod_{j=1}^{g-1}(\psi_{\star_j} - \psi_{\bullet_j})\right).
$$
Codimension $2$ ($i = 1$ in Corollary \ref{Relations}):
$$
2\hat{\lambda}_2 - \hat{\lambda}_1^2 = \frac{1}{2}\mathcal{G}_*^1\left(\psi_{\star_1} - \psi_{\bullet_1}\right).
$$
\end{example}

We can also apply Theorem \ref{MFPC} to Hodge integrals:

\begin{proposition} \label{MFint}
    For all pseudostable indices $(g, n)$,
    $$
    \int_{\overline{\mathcal{M}}^{\ps}_{g, n}}\frac{(1 + \lambda_1 + \dots + \lambda_g)(1 - \lambda_1 + \dots + (-1)^g\lambda_g)}{\prod_{j = 1}^n(1-\psi_j)} =  \sum_{i=0}^g\frac{(-1)^i}{24^ii!}\int_{\overline{\mathcal{M}}_{g-i, n+i}}\frac{1}{\prod_{j=1}^n(1-\psi_j)},
    $$ 
    where the denominator should be intended as a shorthand for the corresponding expansion as a product of geometric series.
\end{proposition}
\begin{proof} We start by applying Theorem \ref{MFPC}:
\begin{equation*}
   \int_{\overline{\mathcal{M}}^{\ps}_{g, n}}\frac{(1 + \lambda_1 + \dots + \lambda_g)(1 - \lambda_1 + \dots + (-1)^g\lambda_g)}{\prod_{j = 1}^n(1-\psi_j)} = \int_{\overline{\mathcal{M}}_{g, n}}\sum_{i=0}^g\frac{1}{i!}\mathcal{G}_*^i\left(\frac{\prod_{j=1}^i(\psi_{\star_j}-\psi_{\bullet_j})}{\prod_{j = 1}^n(1-p_0^\ast(\psi_j))}\right). 
\end{equation*}
    
Notice that in the product $\prod_{j=1}^i(\psi_{\star_j}-\psi_{\bullet_j})$, the integral vanishes unless there is a class $-\psi_{\bullet_j}$ on each flag adjacent to the genus one tail-vertex. This reduces the Hodge integral to 
\begin{align}\label{eq:nroot}
    &\int_{\overline{\mathcal{M}}_{g, n}}\sum_{i=0}^g\frac{1}{i!}\mathcal{G}_*^i\left(\frac{\prod_{j=1}^i(-\psi_{\bullet_j})}{\prod_{j = 1}^n(1-p_0^\ast(\psi_j))}\right) \nonumber\\ 
    = &\sum_{i=0}^g\frac{1}{i!}\left(\int_{\overline{\mathcal{M}}_{1, 1}}-\psi_{1}\right)^i\left(\int_{\overline{\mathcal{M}}_{g-i, n+i}}\frac{1}{\prod_{j=1}^n(1-\psi_j)}\right).
\end{align}
The first integral of (\ref{eq:nroot}) is known to be $-\frac{1}{24}$ (see \cite{vakil2006moduli} or \cite{zvonkine2014introduction}), simplifying (\ref{eq:nroot}) to 
$$
\sum_{i=0}^g\frac{(-1)^i}{24^ii!}\int_{\overline{\mathcal{M}}_{g-i, n+i}}\frac{1}{\prod_{j=1}^n(1-\psi_j)},
$$
completing the proof.
\end{proof}

If we restrict Proposition \ref{MFint} to just a single $\psi$ class, that is, replace 
$
\frac{1}{\prod_{j = 1}^n(1-\psi_j)}$ by $ \frac{1}{1-\psi_1},
$
the resulting Hodge integrals can be easily computed with the string equation \cite{witten1990two}, leading to the following corollary.

\begin{corollary}
    \label{MFintcor}For all pseudostable indices $(g, n)$, we have
    $$\int_{\overline{\mathcal{M}}^{\ps}_{g, n}}\frac{(1 + \lambda_1 + \dots + \lambda_g)(1 - \lambda_1 + \dots + (-1)^g\lambda_g)}{1-\psi_1} = 0.$$
\end{corollary}
\begin{proof} Using Proposition \ref{MFint} implies
\begin{equation*}
    \int_{\overline{\mathcal{M}}^{\ps}_{g, n}}\frac{(1 + \lambda_1 + \dots + \lambda_g)(1 - \lambda_1 + \dots + (-1)^g\lambda_g)}{1-\psi_1} = \sum_{i=0}^g\frac{(-1)^i}{24^ii!}\int_{\overline{\mathcal{M}}_{g-i, n+i}}\frac{1}{1-\psi_1}.
\end{equation*}    
Since the dimension of $\overline{\mathcal{M}}_{g-i, n+i}$ is $3g-2i + n - 3$, the only nonvanishing term in the geometric series expansion of $\frac{1}{1-\psi_1}$ is $\psi_1^{3g-2i+n-3}$, simplifying our work to 
\begin{equation}\label{eq:stringtoo}
\sum_{i=0}^g\frac{(-1)^i}{24^ii!}\int_{\overline{\mathcal{M}}_{g-i, n+i}}\psi_1^{3g-2i+n-3}.
\end{equation}
Using the string equation \cite{witten1990two}, we can reduce the number of marked points in the integral to get the following:
\begin{equation}\label{eq:string}
    \left(\sum_{i=0}^{g-1}\frac{(-1)^i}{24^ii!}\int_{\overline{\mathcal{M}}_{g-i, 1}}\psi_1^{3g-3i-2}\right) + \frac{(-1)^g}{24^gg!}\int_{\overline{\mathcal{M}}_{0, 3}}1.
\end{equation}

We can evaluate the integrals in the summation of (\ref{eq:string}) using an application of the Witten-Kontsevich Theorem for one-pointed psi class intersection numbers \cite{WKT}: 
\begin{equation}\label{eq:onepoint}
    \int_{\overline{\mathcal{M}}_{g-i, 1}} \psi_1^{3(g-i)-2} = \frac{1}{24^{g-i}(g-i)!}.  
\end{equation}
Since $\overline{\mathcal{M}}_{0, 3}$ is a point, then the latter integral of (\ref{eq:string}) is equal to 1. Together with (\ref{eq:onepoint}), this simplifies (\ref{eq:string}) to 
\begin{align}\label{eq:1/24}
    \left(\sum_{i=0}^{g-1}\frac{(-1)^i}{24^ii!} \cdot \frac{1}{24^{g-i}(g-i)!}\right) + \frac{(-1)^g}{24^gg!} &= \left(\sum_{i=0}^{g-1}\frac{(-1)^i}{24^gi!(g-i)!}\right) + \frac{(-1)^g}{24^gg!}\nonumber\\
    &= \left(\sum_{i=0}^{g}\frac{(-1)^i}{24^gi!(g-i)!}\right)\nonumber\\
    &= \frac{1}{24^g}\sum_{i=0}^g\frac{(-1)^i}{i!(g-i)!}\\
    &= 0.\nonumber
\end{align}
\end{proof}

The summation in (\ref{eq:1/24}) highlights another result for specific terms of Corollary \ref{MFintcor}:

\begin{corollary}\label{Cor:Cor}
    For all pseudostable indices $(g, n)$ and $k = 0, \dots, g$, we have $$\int_{\overline{\mathcal{M}}^{\ps}_{g, n}}\left(\sum_{j = \max\{0, 2k-g\}}^{\min\{2k, g\}}(-1)^j\lambda_{2k-j}\lambda_j\right)\psi_1^{3g-2k+n-3} = \frac{(-1)^k}{24^g k!(g-k)!}.$$
\end{corollary}
\begin{proof}
We observe first of all that the Hodge part in the parenthesis is the degree $2k$ part of $(1 + \lambda_1 + \dots + \lambda_g)(1 - \lambda_1 + \dots + (-1)^g\lambda_g)$. Proceeding as in the proof of Corollary \ref{MFintcor} above,
we have
\begin{equation}
   \int_{\overline{\mathcal{M}}^{\ps}_{g, n}}\left(\sum_{j = \max\{0, 2k-g\}}^{\min\{2k, g\}}(-1)^j\lambda_{2k-j}\lambda_j\right)\psi_1^{3g-2k+n-3} = \sum_{i=0}^g\frac{(-1)^i}{24^ii!}\int_{\overline{\mathcal{M}}_{g-i, n+i}}\psi_1^{3g-2k+n-3}.  
\end{equation}
 By dimension reasons the only nonvanishing summand corresponds to $i=k$. We recognize this term as  the $k$-th summand of  
 (\ref{eq:stringtoo}), which proves the result.
\end{proof}

The following example highlights specific cases of Corollary \ref{Cor:Cor}.

\begin{example}\label{ex:MumfordHodge1} For $k = g:$ 
    $$
    \int_{\overline{\mathcal{M}}^{\ps}_{g, n}} \lambda_g^2\psi_1^{g-3+n} = \frac{1}{24^gg!}.
    $$
    For $k = g-1:$ 
    $$
    \int_{\overline{\mathcal{M}}^{\ps}_{g, n}} (2\lambda_g\lambda_{g-2} - \lambda_{g-2}^2)\psi_1^{g-1+n} = \frac{-1}{24^g(g-1)!}.
    $$
    For $k = 1:$
    $$
    \int_{\overline{\mathcal{M}}^{\ps}_{g, n}} (2\lambda_2 - \lambda_1^2)\psi_1^{3g-5+n} = \frac{-1}{24^g(g-1)!}.
    $$
\end{example}

\begin{remark}
We observe a curious duality among the following stable and pseudostable Hodge integrals:
    \begin{equation}\label{eq:lineone}
    \int_{\overline{\mathcal{M}}^{\ps}_{g, n}} \lambda_g^2\psi_1^{g-3+n} =  \int_{\overline{\mathcal{M}}_{g, n}}\frac{(1 + \lambda_1 + \dots + \lambda_g)(1 - \lambda_1 + \dots + (-1)^g\lambda_g)}{1-\psi_1},
    \end{equation}
    \begin{equation}\label{eq:linetwo}
    \int_{\overline{\mathcal{M}}_{g, n}} \lambda_g^2\psi_1^{g-3+n} =   \int_{\overline{\mathcal{M}}^{\ps}_{g, n}}\frac{(1 + \lambda_1 + \dots + \lambda_g)(1 - \lambda_1 + \dots + (-1)^g\lambda_g)}{1-\psi_1}.
    \end{equation}

For $g = 0$ and any value of $n\geq 3$, both lines are giving the  equation $1 = 1$. In higher genus,  \eqref{eq:lineone} evaluates to $\frac{1}{24^gg!}$: we saw the pseudostable integral in Example \ref{ex:MumfordHodge1}, whereas the stable evaluation is obtained from \eqref{eq:onepoint} after applying Mumoford's relation to the numerator. The second line \eqref{eq:linetwo} evaluates to $0$: in the stable case this follows from Mumford's relation $\lambda_g^2 = 0$, whereas the pseudostable integral is Corollary \ref{MFintcor}.

\end{remark}

\section{Quadratic Pseudostable Hodge integrals}\label{sec:quadratic}

In this section, we give a formula for computing all quadratic pseudostable Hodge integrals in terms of their stable counterparts, Theorem \ref{QHI}. We then show in Theorem \ref{Generating} that the formulas  arising from Theorem \ref{QHI} are nicely encoded as a relation among generating functions for quadratic Hodge integrals. 
We begin with a Lemma which will be used to 
simplify some of the algebra in the proof of Theorem \ref{QHI}.

\begin{lemma}\label{lemma}For all non-negative integers $k$,
$$\sum_{(r, s) \in T}\frac{(-1)^{r+s-k}}{(r+s-k)!(k-s)!(k-r)!} = \frac{1}{k!},
$$
where $T\subset \mathbb{Z}^2$ is the set of integer tuples $(r, s)$ with $0 \leq r \leq k, 0 \leq s \leq k,$ and $r + s \geq k$.
\end{lemma}
\begin{proof}
    Note that $T$ can equivalently be defined by all $(r, s)$ where $0 \leq r \leq k$ and $k-r \leq s \leq k$. So the sum is 
    $$
    \sum_{r=0}^k\left(\sum_{s=k-r}^k\frac{(-1)^{r+s-k}}{(r+s-k)!(k-s)!(k-r)!}\right) = \sum_{r=0}^k\frac{1}{(k-r)!}\left(\sum_{s=k-r}^k\frac{(-1)^{r+s-k}}{(r+s-k)!(k-s)!}\right).
    $$
    For convenience, define $N := r+s-k$, which leads to
    $$
    \sum_{r=0}^k\frac{1}{(k-r)!}\left(\sum_{N = 0}^r\frac{(-1)^{N}}{N!(r-N)!}\right).
    $$
    The inner sum is equal to $0$, except when $r = 0$, in which case it is equal to $1$. This allows us to reduce the work to just that case:
    \begin{align*}
        \sum_{r=0}^0\frac{1}{(k-r)!}
        &= \frac{1}{k!},
    \end{align*}
    proving the result.
\end{proof}

\begin{theorem}\label{QHI}
    For all pseudostable indices $(g,n)$,  $i, j = 0, \dots, g$, and for any polynomial $F \in \mathbb{C}[a_1, \dots, a_n]$, $$\int_{\overline{\mathcal{M}}^{\ps}_{g, n}} \lambda_{i}\lambda_{j}F(\psi_1, \dots, \psi_n) = \sum_{k=0}^{\min\{i, j\}} \frac{1}{24^kk!}\int_{\overline{\mathcal{M}}_{g-k, n+k}} \lambda_{i-k}\lambda_{j-k}F(\psi_1, \dots, \psi_n).$$
\end{theorem}

\begin{proof}
From Theorem \ref{Pullback}, we know $$\int_{\overline{\mathcal{M}}^{\ps}_{g, n}} \lambda_{i}\lambda_{j}F(\psi_1, \dots, \psi_n) = \int_{\overline{\mathcal{M}}_{g, n}} \hat{\lambda}_{i}\hat{\lambda}_{j}F(\psi_1, \dots, \psi_n),$$ where 
$$
\hat{\lambda}_{i} = \sum_{r=0}^{i}\frac{1}{r!}\mathcal{G}_*^r(p_0^*(\lambda_{i-r})) \text{ and } \hat{\lambda}_{j} = \sum_{s=0}^{j}\frac{1}{s!}\mathcal{G}_*^s(p_0^*(\lambda_{j-s})).
$$ 

Consider the product
\begin{equation}\label{eq:prodsums}
\hat{\lambda}_{i}\hat{\lambda}_{j} = \left(\sum_{r=0}^{i}\frac{1}{r!}\mathcal{G}_*^r(p_0^*(\lambda_{i-r}))\right)\left(\sum_{s=0}^{j}\frac{1}{s!}\mathcal{G}_*^s(p_0^*(\lambda_{j-s}))\right).
\end{equation}
The result of (\ref{eq:prodsums})  contains many terms supported on strata parameterizing curves with $k$ elliptic tails, where $\max\{r, s\} \leq k \leq \min\{r + s, g\}$. Such terms  occur in the expansion of the term indexed by $(r,s)$ if the integer tuple $(r, s)$ satisfies the following conditions: 
\begin{align}
\label{eq:range}
0 \leq r \leq k, \ \ \ \ \
0 \leq s \leq k, \ \ \ \ \
r +s \geq k,\end{align}
i.e., it belongs to the indexing set $T$ from Lemma \ref{lemma}.
For each pair $(r, s)$ in $T$, we describe the coefficient of the $k$-tails stratum in the following product:
\begin{equation}\label{eq:rs}
    \frac{1}{r!}\mathcal{G}_*^r(p_0^*(\lambda_{i-r}))\cdot\frac{1}{s!}\mathcal{G}_*^s(p_0^*(\lambda_{j-s})).
\end{equation}

Using notation from \cite{Yang}, let $G$ and $H$ be the two decorated graphs representing $\mathcal{G}_*^r(p_0^*(\lambda_{i-r}))$ and $\mathcal{G}_*^s(p_0^*(\lambda_{i-s}))$, respectively, and let $A$ be the unique isomorphism class of decorated graphs in the resulting product supported on a stratum parameterizing curves with $k$ elliptic tails. Our goal is to identify its decorations and numerical coefficient. Of the $k$ edges of $A$, there are $k-s$ edges from only the $G$-structure, $k-r$ edges from only the $H$-structure, and $N := r+s-k$ edges from both the $G$ and $H$-structures. We call these three types of edges $G$-edges, $H$-edges and $(G,H)$-edges, respectively. All the edges, together with their adjacent vertices  have associated decorations. We break the computation in three parts.

\noindent\textsc{Decorations of $(G,H)$-edges.} The decoration for each of these edges is $(-\psi_{\bullet}-\psi_{\star})$, and with there being $N := r+s-k$ such edges, the contribution is 
\begin{equation}
\label{eq:tailco}
\mathcal{G}_*^k\left(\prod_{t=1}^{N}(-\psi_{\bullet_t} - \psi_{\star_t})\right).  
\end{equation}

We observe that when integrating over $\overline{\mathcal{M}}_{g, n}$, any non-vanishing term must have a class $-\psi_{\bullet_t}$ on each flag adjacent to the  genus one tail-vertex, thus the only non-vanishing contribution to the Hodge integral from \eqref{eq:tailco} is

\begin{equation}
(-1)^{N}\mathcal{G}_*^k\left(\prod_{t=1}^{N}\psi_{\bullet_t}\right).
\end{equation}

\noindent\textsc{Decorations of $G$-edges or $H$-edges.} By how the Hodge bundle splits along boundary  strata, the decorations of the tail vertices of these edges are either $1$ or $\lambda_1$. After integrating over $\overline{\mathcal{M}}_{g, n}$, all terms that have edges with a coefficient of $1$ associated to any of the outer vertices vanish by dimension reasons, so we can disregard them here. 
It follows that for the contributing terms, on the root vertex there is the following coefficient:
$$
\lambda_{(i-r)-(k-r)}\lambda_{(j-s)-(k-s)} = \lambda_{i-k}\lambda_{j-k}.
$$ 
Thus the contribution from these edges is
\begin{equation}
    \label{eq:onlyGonlyH}
\mathcal{G}_*^k\left(p_0^*(\lambda_{i-k}\lambda_{j-k})\prod_{t=N+1}^kp_t^\ast(\lambda_1)\right).
\end{equation}

\noindent\textsc{Numerical coefficient.}
There are many different but isomorphic $(G, H)$-structures on the graph $\underline{A}$. One possible way to organize the count is the following.
We first count the possible ways to obtain the $(G, H)$-edges: one must choose $N$ of the $k$ edges of $A$, then  $N$ of the $r$ and $s$ edges of $G$ and $H$, respectively; finally there are $N!$ ways to map the $N$ chosen edges of both $G$  and $H$ to the chosen edges of $A$. This gives us the following  count:
$$
\binom{k}{N}\binom{r}{N}\binom{s}{N}N!N!(-1)^{N},
$$
which simplifies to
$$
\frac{k!r!s!(-1)^{N}}{N!(k-N)!(r-N)!(s-N)!}.
$$
Having dealt with the $(G, H)$-edges, now we have to choose $r-N$ of the $k-N$ remaining edges of $A$ to create subsets of size $r-N$ and $s-N$; then there are $(r-N)!$ and $(s-N)!$ bijections of the chosen subsets of $A$ into the chosen subsets of $G$ and $H$, respectively. This gives us the following count:
$$
\binom{k-N}{r-N}(r-N)!(s-N)!,
$$
which simplifies to 
$
(k-N)!.
$
Multiplying the two, we obtain the total number of $(G,H)$-structures on $\underline{A}$.
The numerical coefficient for $\underline{A}$ in the multiplication \eqref{eq:rs} is given by
\begin{equation}
   \frac{1}{r!}\cdot \frac{1}{s!}\cdot \frac{1}{k!}\cdot\frac{k!r!s!(-1)^{N}(k-N)!}{N!(k-N)!(r-N)!(s-N)!} = \frac{(-1)^{N}}{N!(r-N)!(s-N)!},
\end{equation}
where the factor of $\frac{1}{k!}$ is from $|\text{Aut}(A)| = k!$, following the algorithm for multiplication of graphs $G$ and $H$ shown in \cite{Yang}.

\noindent\textsc{Total contribution.}
Combining all types of decorations, and the numerical coefficient, the (isomorphism class of the) decorated graph  $A$ is given by 
\begin{equation}
  \label{eq:quasithere}  
\frac{(-1)^{N}}{N!(r-N)!(s-N)!}\mathcal{G}_*^k\left(p_0^*(\lambda_{i-k}\lambda_{j-k})\prod_{t=1}^{N}\psi_{\bullet_t}\prod_{t=N+1}^kp_t^\ast(\lambda_1)\right).
\end{equation}

Multiplying \eqref{eq:quasithere} by $F(\psi_1, \ldots, \psi_n)$ and integrating  over $\overline{\mathcal{M}}_{g, n}$ yields
\begin{align}\label{eq:co}
 \nonumber   &\frac{(-1)^N}{N!(r-N)!(s-N)!}\cdot\int_{\overline{\mathcal{M}}_{g, n}}\mathcal{G}_*^k\left(p_0^*(\lambda_{i-k}\lambda_{j-k})\prod_{t=1}^{N}\psi_{\bullet_t}\prod_{t=N+1}^k(p_t^\ast(\lambda_1))\right)F(\psi_1, \ldots, \psi_n)\\
 \nonumber   = &\frac{(-1)^N}{N!(r-N)!(s-N)!}\cdot\left(\int_{\overline{\mathcal{M}}_{g-k, n+k}} \lambda_{i-k}\lambda_{j-k}F(\psi_1, \ldots, \psi_n)\right)\left(\int_{\overline{\mathcal{M}}_{1, 1}}\psi_{1}\right)^{N}\left(\int_{\overline{\mathcal{M}}_{1, 1}}\lambda_1\right)^{k-N}\\
 \nonumber   = &\frac{\left(-1\right)^{r+s-k}}{24^k(r+s-k)!(k-s)!(k-r)!}\int_{\overline{\mathcal{M}}_{g-k, n+k}}\lambda_{i-k}\lambda_{j-k} F(\psi_1, \ldots, \psi_n).\\
  & 
\end{align}

The contribution to the Hodge integral by the $k$-tail stratum, is obtained by  adding  \eqref{eq:co} for the range of pairs $(r,s)$ from \eqref{eq:range}:

$$
\sum_{(r, s) \in T} \frac{\left(-1\right)^{r+s-k}}{24^k(r+s-k)!(k-s)!(k-r)!}\int_{\overline{\mathcal{M}}_{g-k, n+k}}\lambda_{i-k}\lambda_{j-k}F(\psi_1, \dots, \psi_n),
$$
where $T$ is the set of integer tuples $(r, s)$ with $0 \leq r \leq k, 0 \leq s \leq k,$ and $r + s \geq k$. Using Lemma \ref{lemma}, this sum simplifies to 
$$
\frac{1}{24^kk!}\int_{\overline{\mathcal{M}}_{g-k, n+k}}\lambda_{i-k}\lambda_{j-k}F(\psi_1, \dots, \psi_n).
$$
Thus we have $$\int_{\overline{\mathcal{M}}_{g, n}} \hat{\lambda}_{i}\hat{\lambda}_{j}F(\psi_1, \dots, \psi_n) = \sum_{k=0}^{\min\{i, j\}}\frac{1}{24^kk!}\int_{\overline{\mathcal{M}}_{g-k, n+k}}\lambda_{i-k}\lambda_{j-k}F(\psi_1, \dots, \psi_n),$$ where the upper bound on the sum of $\min\{i, j\}$ is due to $\lambda_{i-k} = 0$ and $\lambda_{j-k} = 0$ when $i - k < 0$ and $j - k < 0,$ respectively. 
\end{proof}

\begin{sloppypar}
We now present a comparison between all pseudostable and stable quadratic Hodge integrals using generating functions. While we used $\lambda_i$ and $\lambda_j$ in Theorem \ref{QHI}, we instead use $\lambda_{g-i}$ and $\lambda_{g-j}$ in the generating functions for convenience.
\end{sloppypar}

\begin{theorem}\label{Generating} Let $x,y,z,t_m$, for $m$ any positive integer, be formal variables. Define
$$ 
\mathcal{F}^{ps}(x, y, z, \{t_m\}):=\sum_{i, j, g, n = 0}^\infty x^iy^jz^g \int_{\overline{\mathcal{M}}^{\ps}_{g, n}} \frac{\lambda_{g-i}\lambda_{g-j}}{\prod_{m=1}^n (1 - t_m\psi_m)},
$$
where the denominator should be intended as a shorthand for the corresponding expansion as a product of geometric series. Define $\mathcal{F}(x, y, z, \{t_m\})$ similarly for stable Hodge integrals. In both cases we adopt the convention that indices $(g,n)$ for which we have not defined a corresponding moduli space correspond to zero summands. Then: 
$$
\mathcal{F}^{ps}(x, y, z, \{t_m\}) = e^{z/24}\mathcal{F}(x, y, z, \{t_m\}).
$$
\end{theorem}

\begin{proof}
Let $\underline{\mu} = (\mu_1, \ldots, \mu_n)$ be a vector of non-negative integers and denote by $t^{\underline{\mu}}$ the monomial $\prod_{m = 1}^n t_m^{\mu_m}$.
Given  $i, j, g, \underline{\mu}$, the coefficient of the monomial $x^iy^jz^g t^{\underline{\mu}}$ in $\mathcal{F}^{ps}(x, y, z, \{t_m\})$ is equal to
\begin{equation}\label{Eq:firstside}
\int_{\overline{\mathcal{M}}^{\ps}_{g, n}} \lambda_{g-i}\lambda_{g-j}\prod_{m = 1}^n \psi_m^{\mu_m},
\end{equation}
and is non-zero only if $|\underline{\mu}| = g + i + j + n - 3$. Note that the same condition holds in the stable case. 
We now compute the coefficient of the same monomial in the 
expression $e^{z/24}\mathcal{F}(x, y, z, \{t_m\})$. We first expand the product to obtain:
    \begin{align*}
        e^{z/24}\mathcal{F}(x, y, z, \{t_m\})
        = &\left(\sum_{g = 0}^\infty\frac{1}{24^gg!}z^g\right)\left(\sum_{i, j, g, n = 0}^\infty x^iy^jz^g \int_{\overline{\mathcal{M}}_{g, n}} \frac{\lambda_{g-i}\lambda_{g-j}}{\prod_{m=1}^n (1 - t_m\psi_m)}\right)\\
        = &\sum_{i, j, g, n=0}^\infty x^iy^jz^g\left(\sum_{k=0}^g \frac{1}{24^kk!}\int_{\overline{\mathcal{M}}_{g-k, n}} \frac{\lambda_{(g-k)-i}\lambda_{(g-k)-j}}{\prod_{m=1}^n (1 - t_m\psi_m)}\right).
    \end{align*}

The coefficient of the monomial $x^iy^jz^g t^{\underline{\mu}}$ is 
\begin{equation}\label{eq:otherside}
\sum_{k=0}^{\min\{g-i, g-j\}} \frac{1}{24^kk!}\int_{\overline{\mathcal{M}}_{g-k, n+k}} {\lambda_{(g-i)-k}\lambda_{(g-j)-k}}\prod_{m = 1}^n \psi_m^{\mu_m}
\end{equation}
    
    The reason for $n+k$ instead of $n$ is because that is the only number of marked points which  returns a nonzero integral for dimension reasons. Further, notice that while the upper bound of $k$ in the inner sum is $g$, any value of $k$ greater than $\min\{g-i, g-j\}$ will make $\lambda_{(g-k)-i}$ or $\lambda_{(g-k)-j}$ equal to zero, so we can instead use $\min\{g-i, g-j\}$ for the upper bound.
   
   We now observe that the equality of \eqref{Eq:firstside} and \eqref{eq:otherside} is precisely the statement of  Theorem \ref{QHI}, thus concluding the proof.
   
\end{proof}

\subsection{Applications} We conclude the paper with two applications: first we observe the how simple and explicit the equality of generating function from Theorem \ref{Generating} is when restricting attention to Hodge integrals with a $\lambda_g^2$ term and a single $\psi$ class. Next we give a closed formula the evaluation of arbitrary $\lambda_g\lambda_{g-1}$ pseudostable Hodge integrals.

\begin{example}
    We  demonstrate Theorem \ref{Generating} for $x = 0, y = 0,$ and $\{t_m\} = (t, 0, 0, 0, \dots)$. That is, we  explicitly observe \begin{equation}\label{eq:generating}
        \sum_{g, n = 0}^\infty z^g \int_{\overline{\mathcal{M}}^{\ps}_{g, n}} \frac{\lambda_{g}^2}{1 - t\psi_1} = e^{z/24}\left(\sum_{g, n = 0}^\infty z^g \int_{\overline{\mathcal{M}}_{g, n}} \frac{\lambda_{g}^2}{1 - t\psi_1}\right).
    \end{equation}

    Starting on the left side of (\ref{eq:generating}), the only term of the integrand which will give a nonzero result is $\lambda_g^2(t\psi_1)^{g+n-3},$ so 

    \begin{align}\label{eq:genpseudo}
        \sum_{g, n = 0}^\infty z^g \int_{\overline{\mathcal{M}}^{\ps}_{g, n}} \frac{\lambda_{g}^2}{1 - t\psi_1} = \sum_{g, n = 0}^\infty z^g t^{g+n-3}\int_{\overline{\mathcal{M}}^{\ps}_{g, n}} \lambda_g^2\psi_1^{g+n-3}.
    \end{align}

    Looking at the right side of (\ref{eq:generating}), note that $\lambda_g^2 = 0$ except when $g = 0$, in which case $\lambda_0^2 = 1$. The only term in the expansion of $\frac{1}{1-t\psi_1}$ which  produces a nonzero result is $(t\psi_1)^{g+n-3}$, implying
    \begin{align}\label{eq:genstring}
        e^{z/24}\left(\sum_{g, n = 0}^\infty z^g \int_{\overline{\mathcal{M}}_{g, n}} \frac{\lambda_{g}^2}{1 - t\psi_1}\right) &= e^{z/24}\left(\sum_{n = 3}^\infty t^{n-3}\int_{\overline{\mathcal{M}}_{0, n}} \psi_1^{n-3} \right)
        = e^{z/24}\sum_{n = 3}^\infty t^{n-3}.
    \end{align}
    The last equality of (\ref{eq:genstring}) used the string equation \cite{witten1990two} to evaluate each of the integrals to $1$.
By Theorem \ref{Generating}, 
\begin{equation} \label{eq:final}\sum_{g, n = 0}^\infty
    \left(\int_{\overline{\mathcal{M}}^{\ps}_{g, n}} \lambda_g^2\psi_1^{g+n-3}\right)z^g t^{g+n-3} = \frac{e^{z/24}}{1-t}.\end{equation}
    The coefficient of $z^gt^{g+n-3}$ in \eqref{eq:final} is $\frac{1}{24^g}\cdot\frac{1}{g!}$,  which aligns with what we found in Example \ref{ex:MumfordHodge1}, where we computed these integrals explicitly.
\end{example}

\begin{proposition}
    For all pseudostable indices $(g, n)$, 
    $$
    \int_{\overline{\mathcal{M}}^{\ps}_{g, n}} \psi_1^{d_1}\cdots \psi_n^{d_n}\lambda_g\lambda_{g-1} = \sum_{k=0}^{g-1}\frac{1}{24^kk!}\cdot\frac{(2(g-k)-3+(n+k))!|B_{2(g-k)}|}{2^{2(g-k)-1}(2(g-k))!\prod_{j=1}^{n+k}(2d_j-1)!!},
    $$
    where $B_{2(g-k)}$ represents the $2(g-k)$-th Bernoulli number.
\end{proposition}

\begin{proof}
    Applying Theorem \ref{QHI} to the quadratic pseudostable Hodge integral gives $$\int_{\overline{\mathcal{M}}^{\ps}_{g, n}} \psi_1^{d_1}\cdots \psi_n^{d_n}\lambda_g\lambda_{g-1} = \sum_{k=0}^{g-1}\frac{1}{24^kk!}\int_{\overline{\mathcal{M}}_{g-k, n+k}} \psi_1^{d_1}\cdots\psi_n^{d_n}\lambda_{g-k}\lambda_{(g-k)-1}.$$ Each of the integrals in the sum can be evaluated with the Faber intersection number conjecture (conjectured in \cite{FINC}, proven in \cite{FINCpf}):
    $$
    \int_{\overline{\mathcal{M}}_{g, n}} \psi_1^{d_1}\cdots\psi_n^{d_n}\lambda_{g}\lambda_{g-1} = \frac{(2g - 3 + n)!|B_{2g}|}{2^{2g-1}(2g)!\prod_{j=1}^n(2d_j-1)!!},
    $$
    where $B_{2g}$ is the $2g$-th Bernoulli number. Replace $g$ with $g-k$, and this yields the desired result.
\end{proof}
\appendix

\bibliographystyle{alpha}

\bibliography{references}

\newcommand{\etalchar}[1]{$^{#1}$}
\def\cprime{$'$}
\begin{thebibliography}{HKK{\etalchar{+}}03}

\bibitem[AC96]{ac:cag}
Enrico Arbarello and Maurizio Cornalba.
\newblock Combinatorial and algebro-geometric cohomology classes on the moduli spaces of curves.
\newblock {\em J. Algebraic Geom.}, 5(4):705--749, 1996.

\bibitem[AM93]{am:top}
Paul~S. Aspinwall and David~R. Morrison.
\newblock Topological field theory and rational curves.
\newblock {\em Comm. Math. Phys.}, 151(2):245--262, 1993.

\bibitem[CGR{\etalchar{+}}22]{PSHI}
Renzo Cavalieri, Joel Gallegos, Dustin Ross, Brandon Van~Over, and Jonathan Wise.
\newblock {Pseudostable Hodge Integrals}.
\newblock {\em International Mathematics Research Notices}, 09 2022.
\newblock rnac253.

\bibitem[CK99]{ck:mirror}
David~A. Cox and Sheldon Katz.
\newblock {\em Mirror symmetry and algebraic geometry}, volume~68 of {\em Mathematical Surveys and Monographs}.
\newblock American Mathematical Society, Providence, RI, 1999.

\bibitem[CMW24]{CM:pseudo}
Renzo Cavalieri, Steffen Marcus, and Jonathan Wise.
\newblock Tropical pseudostable curves.
\newblock {\em arXiv:2404.02654}, 2024.

\bibitem[DM69]{DMspace}
Pierre Deligne and David Mumford.
\newblock The irreducibility of the space of curves of given genus.
\newblock {\em Publications Math{\'e}matiques de l'IHES}, 36:75--109, 1969.

\bibitem[DSvZ22]{delecroix2022admcycles}
Vincent Delecroix, Johannes Schmitt, and Jason van Zelm.
\newblock admcycles-a sage package for calculations in the tautological ring of the moduli space of stable curves.
\newblock {\em Journal of Software for Algebra and Geometry}, 11(1):89--112, 2022.

\bibitem[Ele02]{ele:am}
Artur Elezi.
\newblock The {A}spinwall-{M}orrison calculation and {G}romov-{W}itten theory.
\newblock {\em Pacific J. Math.}, 205(1):99--108, 2002.

\bibitem[Fab99]{FINC}
Carel Faber.
\newblock A conjectural description of the tautological ring of the moduli space of curves.
\newblock In {\em Moduli of curves and abelian varieties}, pages 109--129. Springer, 1999.

\bibitem[FP00a]{faber2000hodge}
Carel Faber and Rahul Pandharipande.
\newblock Hodge integrals and gromov-witten theory.
\newblock {\em Inventiones mathematicae}, 1(139):173--199, 2000.

\bibitem[FP00b]{faber2000logarithmic}
Carel Faber and Rahul Pandharipande.
\newblock Logarithmic series and hodge integrals in the tautological ring. with an appendix by don zagier.
\newblock {\em Michigan Mathematical Journal}, 48(1):215--252, 2000.

\bibitem[GP99]{gp:virtloc}
T.~Graber and R.~Pandharipande.
\newblock Localization of virtual classes.
\newblock {\em Invent. Math.}, 135(2):487--518, 1999.

\bibitem[GP03]{graberpandharipande}
Tom Graber and Rahul Pandharipande.
\newblock Constructions of nontautological classes on moduli spaces of curves.
\newblock {\em Michigan Math Journal}, 51:9, 2003.

\bibitem[HH09]{HH}
Brendan Hassett and Donghoon Hyeon.
\newblock Log canonical models for the moduli space of curves: The first divisorial contraction.
\newblock {\em Transactions of the American Mathematical Society}, 361(8):4471--4489, 2009.

\bibitem[HKK{\etalchar{+}}03]{clay:mirr}
Kentaro Hori, Sheldon Katz, Albrecht Klemm, Rahul Pandharipande, Richard Thomas, Cumrun Vafa, Ravi Vakil, and Eric Zaslow.
\newblock {\em Mirror symmetry}, volume~1 of {\em Clay Mathematics Monographs}.
\newblock American Mathematical Society, Providence, RI; Clay Mathematics Institute, Cambridge, MA, 2003.
\newblock With a preface by Vafa.

\bibitem[Kon92]{WKT}
Maxim Kontsevich.
\newblock {Intersection theory on the moduli space of curves and the matrix Airy function}.
\newblock {\em Communications in Mathematical Physics}, 147(1):1--23, 1992.

\bibitem[LX09]{FINCpf}
Kefeng Liu and Hao Xu.
\newblock {A proof of the Faber intersection number conjecture}.
\newblock {\em Journal of Differential Geometry}, 83(2):313--335, 2009.

\bibitem[Mor87]{morita84}
Shigeyuki Morita.
\newblock Characteristic classes of surface bundles.
\newblock {\em Invent. Math.}, 90(3):551--577, 1987.

\bibitem[Mum83]{Mumford1983}
David Mumford.
\newblock {\em Towards an Enumerative Geometry of the Moduli Space of Curves}, pages 271--328.
\newblock Birkh{\"a}user Boston, Boston, MA, 1983.

\bibitem[Sch91]{schubert1991new}
David Schubert.
\newblock A new compactification of the moduli space of curves.
\newblock {\em Compositio Mathematica}, 78(3):297--313, 1991.

\bibitem[Vak06]{vakil2006moduli}
Ravi Vakil.
\newblock The moduli space of curves and gromov-witten theory.
\newblock {\em arXiv:math/0602347}, 2006.

\bibitem[Wit90]{witten1990two}
Edward Witten.
\newblock Two-dimensional gravity and intersection theory on moduli space.
\newblock {\em Surveys in differential geometry}, 1(1):243--310, 1990.

\bibitem[Yan10a]{Yang}
Stephanie Yang.
\newblock Calculating intersection numbers on moduli spaces of curves.
\newblock {\em arXiv:0808.1974}, 2010.

\bibitem[Yan10b]{yang2010intersection}
Stephanie Yang.
\newblock Intersection numbers on {$\overline{\mathcal{M}}_{g, n}$}.
\newblock {\em Journal of Software for Algebra and Geometry}, 2(1):1--5, 2010.

\bibitem[Zvo14]{zvonkine2014introduction}
Dimitri Zvonkine.
\newblock An introduction to moduli spaces of curves and their intersection theory.
\newblock {\em Le{\c{c}}ons des journ{\'e}es math{\'e}matiques de Glanon}, 2014.

\end{thebibliography}

\end{document}